\documentclass[12pt]{article}

\usepackage{verbatim}
\usepackage{array,multirow}
\usepackage{fullpage}
\usepackage{enumitem}
\usepackage{amssymb,amsfonts,amsmath,amsthm}
\usepackage{url}
\usepackage{tikz}
\usepackage[colorlinks=true,linkcolor=magenta,citecolor=olive]{hyperref}
\usepackage{ctable}

\theoremstyle{plain}
\newtheorem{theorem}{Theorem}[section]
\newtheorem{lemma}[theorem]{Lemma}

\newtheorem{proposition}[theorem]{Proposition}
\newtheorem{claim}[theorem]{Claim}

\newtheorem{definition}[theorem]{Definition}
\newtheorem{fact}[theorem]{Fact}

\numberwithin{equation}{section} 

\newcommand{\etal}{\textit{et~al. }}
\def\SD{\ensuremath{\mathsf{SD}}}

\def\leftson{\ensuremath{\mathsf{leftSon}} }
\def\rightson{\ensuremath{\mathsf{rightSon}} }
\def\parent{\ensuremath{\mathsf{parent}} }
\def\root{\ensuremath{\mathsf{root}} }
\def\dep{\ensuremath{\mathsf{depth}} }

\def\xf{\ensuremath{X_{\mathsf{F}} }}
\def\yf{\ensuremath{Y_{\mathsf{F}} }}
\def\xfi{\ensuremath{X_{\mathsf{FI}} }}
\def\yfi{\ensuremath{Y_{\mathsf{FI}} }}

\def\kin{\ensuremath{k_{\text{in}}} }
\def\kout{\ensuremath{k_{\text{out}}} }

\def\crm{\ensuremath{\mathsf{Resp}} }
\def\dreamresp{\ensuremath{\mathsf{DreamResp}} }
\def\disp{\ensuremath{\mathsf{Disp}} }
\def\subext{\ensuremath{\mathsf{SubExt}} }

\def\ibot{\ensuremath{i_\text{bot}} }
\def\imid{\ensuremath{i_\text{mid}} }

\def\itop{\ensuremath{i_\text{top}} }

\def\vvtop{\ensuremath{v_\text{top}} }
\def\vvmid{\ensuremath{v_\text{mid}} }
\def\obsvvmid{\ensuremath{v^{\text{observed}}_\text{mid}} }
\def\vvbot{\ensuremath{v_\text{bot}} }
\def\uutop{\ensuremath{u_\text{top}} }
\def\uumid{\ensuremath{u_\text{mid}} }
\def\uubot{\ensuremath{u_\text{bot}} }

\def\pobs{\ensuremath{p_\text{observed}} }
\def\qobs{\ensuremath{q_\text{observed}} }

\def\lefty{\ensuremath{\mathsf{left}} }
\def\righty{\ensuremath{\mathsf{right}} }

\def\ch{\ensuremath{\mathsf{Challenge}} }
\def\resp{\ensuremath{\mathsf{Response}} }

\def\newch{\ensuremath{\mathsf{NodePathCh}} }

\def\f{\ensuremath{\mathsf{F}} }
\def\bs{\ensuremath{\mathsf{B}} }
\def\btop{\ensuremath{\mathsf{B}_\text{top}} }
\def\bmid{\ensuremath{\mathsf{B}_\text{mid}} }
\def\bbot{\ensuremath{\mathsf{B}_\text{bot}} }

\def\h{\ensuremath{\mathsf{H}} }

\def\lbl{\ensuremath{\mathsf{label}} }
\def\er{\ensuremath{2^{-k^{\Omega(1)}}} }

\def\se{\ensuremath{\mathsf{SE}} }
\def\bext{\ensuremath{\mathsf{BExt}} }

\def\fixed{\ensuremath{\mathsf{fixed}} }
\def\hasent{\ensuremath{\mathsf{hasEntropy}} }

\def\sup{\ensuremath{\mathsf{supp}} }

\def\poly{{\mathrm{poly}}}

\def\polylog{{\mathrm{polylog}}}

\newcommand{\eps}{\varepsilon}

\newcommand{\ext}{\mathsf{Ext}}

\newcommand{\me}{H_\infty}

\renewcommand{\Pr}{\mathop{\bf Pr\/}}

\def\pr{\Pr}

\def\poly{{\mathrm{poly}}}
\def\polylog{{\mathrm{polylog}}}

\def\({\left(}
\def\){\right)}

\begin{document}

\title{Two-Source Dispersers for Polylogarithmic Entropy and Improved Ramsey Graphs}
\author{
  Gil Cohen\thanks{
    Department of Computer Science and Applied Mathematics,
    Weizmann Institute of Science, Rehovot, Israel.
    Email: {\tt coheng@gmail.com}.
    Partially supported by an ISF grant and
    by the I-CORE program of the planning and budgeting committee.
}
}

\maketitle

\begin{abstract}
In his 1947 paper that inaugurated the probabilistic method, Erd{\"o}s~\cite{Erdos47} proved the existence of $2\log{n}$-Ramsey graphs on $n$ vertices. Matching Erd{\"o}s' result with a constructive proof is a central problem in combinatorics, that has gained a significant attention in the literature. The state of the art result was obtained in the celebrated paper by Barak, Rao, Shaltiel and Wigderson [Ann. Math'12], who constructed a $2^{2^{(\log\log{n})^{1-\alpha}}}$-Ramsey graph, for some small universal constant $\alpha > 0$.

In this work, we significantly improve the result of Barak~\etal and construct $2^{(\log\log{n})^c}$-Ramsey graphs, for some universal constant $c$. In the language of theoretical computer science, our work resolves the problem of explicitly constructing two-source dispersers for polylogarithmic entropy.
\end{abstract}

\thispagestyle{empty}
\newpage

\small
\tableofcontents
\normalsize
\thispagestyle{empty}
\newpage
\pagenumbering{arabic}

\section{Introduction}

Ramsey theory is a branch of combinatorics that studies the unavoidable presence of local structure in globally-unstructured objects. In the paper that pioneered this field of study, Ramsey~\cite{Ramsey1928} considered an instantiation of this phenomena in graph theory.

\begin{definition}[Ramsey graphs]
A graph on $n$ vertices is called $k$-Ramsey if it contains no clique or independent set of size $k$. \end{definition}

Ramsey showed that there does not exist a graph on $n$ vertices that is $\log(n)/2$-Ramsey. In his influential paper that inaugurated the probabilistic method, Erd{\"o}s~\cite{Erdos47} complemented Ramsey's result and showed that most graphs on $n$ vertices are $2\log{n}$-Ramsey. Unfortunately, Erd{\"o}s' argument is non-constructive and one does not obtain from Erd{\"o}s' proof an example of a graph that is $2\log{n}$-Ramsey. In fact, Erd{\"o}s offered a \$100 dollar prize for matching his result, up to any multiplicative constant factor, by a constructive proof. That is, coming up with an explicit construction of an $O(\log{n})$-Ramsey graph. Erd{\"o}s' challenge gained a significant attention in the literature as summarized in Table~\ref{tab:summary}.

Explicitness got a new meaning in the computational era. While, classically, a succinct mathematical formula was considered to be an explicit description, complexity theory suggests a more relaxed, and arguably more natural interpretation of explicitness. An object is deemed explicit if one can efficiently construct that object from scratch. More specifically, a graph on $n$ vertices is explicit if given the labels of any two vertices $u,v$, one can efficiently determine whether there is an edge connecting $u,v$ in the graph. Since the description length of $u,v$ is $2\log{n}$ bits, quantitatively, by efficient we require that the running-time is $\polylog(n)$.

Ramsey graphs have an analogous definition for bipartite graphs. A bipartite graph on two sets of $n$ vertices is a bipartite $k$-Ramsey if it has no $k \times k$ complete or empty bipartite subgraph. One can show that a bipartite Ramsey graph induces a Ramsey graph with comparable parameters. Thus, constructing bipartite Ramsey-graphs is at least as hard as constructing Ramsey graphs, and it is believed to be a strictly harder problem. Furthermore, Erd{\"o}s' argument holds as is for bipartite graphs.

The main result of this paper is an explicit construction of bipartite Ramsey graphs that significantly improves previous results.

\begin{table}[h]\centering
\renewcommand{\arraystretch}{1.5}
\begin{tabular}{ | m{7cm} | m{4cm} | m{1.8cm} | }
\hline Construction & $k(n)$    & Bipartite \\ \hline
\specialrule{0.15em}{.025em}{.025em}
  \cite{Erdos47} (non-constructive) & $2 \log{n}$        & \quad\,\,\, $\checkmark$      \\ \hline
  \cite{abbott72}  & $n^{\log{2}/\log{5}}$        & \quad\,\,\,       \\ \hline
  \cite{nagy1975}  & $n^{1/3}$        & \quad\,\,\,       \\ \hline
  \cite{frankl1977}  & $n^{o(1)}$        & \quad\,\,\,       \\ \hline
  \cite{chung1981}  & $2^{O((\log{n})^{3/4} \cdot (\log\log{n})^{1/4})}$        & \quad\,\,\,       \\ \hline
  \cite{FW81,naor1992constructing,Alon1998shannon,grolmusz2001low,barak06ramsey}  & $2^{O(\sqrt{\log{n} \cdot \log\log{n}})}$        & \quad\,\,\,       \\ \hline

  The Hadamard matrix (folklore)                   & $n/2$             & \quad\,\,\, $\checkmark$      \\ \hline
  \cite{PR04}                   & $n/2 - \sqrt{n}$             & \quad\,\,\, $\checkmark$      \\ \hline
  \cite{BKSSW10}                    & $o(n)$             & \quad\,\,\, $\checkmark$      \\ \hline
  \cite{brsw06}                     & $2^{2^{(\log\log{n})^{1-\alpha}}} = n^{o(1)}$ &\quad\,\,\,  $\checkmark$ \\ \hline
\specialrule{0.1em}{.015em}{.015em}
  This work                         & $2^{(\log\log{n})^{O(1)}}$      & \quad\,\,\, $\checkmark$ \\ \hline
\end{tabular}
\caption{Summary of constructions of Ramsey graphs.}
\label{tab:summary}
\end{table}

\begin{theorem}[Ramsey graphs]\label{thm:main intro}
There exists an explicit bipartite $2^{(\log\log{n})^{O(1)}}$-Ramsey graph on $n$ vertices.
\end{theorem}

In fact, the graph that we construct has a stronger property. Namely, for $k = 2^{(\log\log{n})^{O(1)}}$, any $k$ by $k$ bipartite subgraph has a relatively large subgraph of its own that has density close to $ 1/2$.

\subsection{Two-source dispersers, extractors, and sub-extractors}

In the language of theoretical computer science, Theorem~\ref{thm:main intro} yields a two-source disperser for polylogarithmic entropy. We first recall some basic definitions.

\begin{definition}[Statistical distance]
The \emph{statistical distance} between two distributions $X,Y$ on a common domain $D$ is defined by
$$
\SD\left(X,Y\right) = \max_{A \subseteq D}\left\{\left|\, \Pr[X \in A] - \Pr[Y \in A]\,\right| \right\}.
$$
If $\SD(X,Y) \le \eps$ we say that $X$ is $\eps$-close to $Y$.
\end{definition}

\begin{definition}[Min-entropy]
The \emph{min-entropy} of a random variable $X$ is defined by
$$
\me(X) = \min_{x \in \sup(X)}{\log_2\left(\frac1{\Pr[X=x]}\right)}.
$$
If $X$ is supported on $\{0,1\}^n$, we define the \emph{min-entropy rate} of $X$ by $\me(X)/n$. In such case, if $X$ has min-entropy $k$ or more, we say that $X$ is an $(n,k)$-weak-source or simply an $(n,k)$-source.
\end{definition}

\begin{definition}[Two-source zero-error dispersers]
A function $\disp \colon \{0,1\}^n \times \{0,1\}^n \to \{0,1\}^m$  is called a \emph{two-source zero-error disperser} for entropy $k$ if for any two independent $(n,k)$-sources $X,Y$, it holds that $$
\sup(\disp(X,Y)) = \{0,1\}^m.
$$
\end{definition}

Note that a two-source zero-error disperser for entropy $k$, with a single output bit, is equivalent to a bipartite $2^k$-Ramsey graph on $2^n$ vertices on each side. Constructing two-source dispersers for polylogarithmic entropy is considered a central problem in pseudorandomness, that we resolve in this paper. Indeed, a $2^{\poly(\log\log{n})}$-Ramsey graph on $n$ vertices is equivalent to a disperser for entropy $\polylog(n)$. From the point of view of dispersers, it is easier to see how challenging is Erd{\"o}s' goal of constructing $O(\log{n})$-Ramsey graphs. Indeed, these are equivalent to dispersers for entropy $\log(n)+O(1)$. Even a disperser for entropy $O(\log{n})$ does not meet Erd{\"o}s' goal as it translates to a $\polylog(n)$-Ramsey graph.

While Theorem~\ref{thm:main intro} already yields a two-source zero-error disperser for polylogarithmic entropy, we can say something stronger.

\begin{theorem}[Two-source zero-error dispersers]\label{thm:main intro dispersers}
There exists an explicit two-source zero-error disperser for $n$-bit sources having entropy $k = \polylog(n)$, with $m = k^{\Omega(1)}$ output bits.
\end{theorem}

Theorem~\ref{thm:main intro dispersers} gives an explicit zero-error disperser for polylogarithmic entropy, with many output bits. In fact, we prove a stronger statement than that. To present it, we recall the notion of a two-source extractor, introduced by Chor and Goldreich~\cite{CG88}.

\begin{definition}[Two-source extractors]
A function $\ext \colon \{0,1\}^n \times \{0,1\}^n \to \{0,1\}^m$  is called a \emph{two-source extractor} for entropy $k$, with error $\eps$, if for any two independent $(n,k)$-sources $X,Y$, it holds that $\ext(X,Y)$ is $\eps$-close to uniform.
\end{definition}

Chor and Goldreich~\cite{CG88} proved that there exist two-source extractors with error $\eps$ for entropy $k = \log(n) + 2 \log(1/\eps) + O(1)$ with $m = 2k - 2 \log(1/\eps) - O(1)$ output bits. A central open problem in pseudorandomness is to match this existential proof with an explicit construction having comparable parameters. Unfortunately, even after almost $30$ years, little progress has been made.

Already in their paper, Chor and Goldreich gave an explicit construction of a two-source extractor for entropy $0.51n$, which is very far from what is obtained by the existential argument. Nevertheless, it took almost $20$ years before any improvement was made. Bourgain~\cite{Bourgain05} constructed a two-source extractor for entropies $(1/2-\alpha) \cdot n$, where $\alpha > 0$ is some small universal constant. An incomparable result was obtained by Raz~\cite{Raz05}, who required one source to have min-entropy $0.51n$ but the other source can have entropy $O(\log{n})$.


In this paper we construct a pseudorandom object that is stronger than a two-source zero-error disperser, yet is weaker than a two-source extractor. Informally speaking, this is a function with the following property. In any two independent weak-sources, there exist two independent weak-sources with comparable amount of entropy to the original sources, restricted to which, the function acts as a two-source extractor. To give the formal definition we first recall the definition of a subsource, introduced in~\cite{BKSSW10}.

\begin{definition}[Subsource]
Given random variables $X$ and $X'$ on $\{0,1\}^n$, we say that $X'$ is a \emph{deficiency $d$} subsource of $X$ and write $X' \subset X$ if there exists a set $A \subseteq \{0,1\}^n$ such that $(X \mid A) = X'$ and $\pr[X \in A] \ge 2^{-d}$. More precisely, for every $a \in A$, $\pr[X' = a]$ is defined by $\pr[X = a \mid X \in A]$ and for $a \not\in A$, $\pr[X' = a] = 0$.
\end{definition}

It is instructive to think of a weak-source $X$ as a random variable that is uniformly distributed over some set $S(X)$. In this case, $X'$ is a subsource of $X$ is the same as saying that $S(X')$ is a subset of $S(X)$. The deficiency determines the density of $S(X')$ in $S(X)$.

\begin{definition}[Two-source sub-extractors]
A function
$$
\subext \colon \{0,1\}^n \times \{0,1\}^n \to \{0,1\}^m
$$
is called a \emph{two-source sub-extractor} for outer-entropy $\kout$ and inner-entropy $\kin$, with error $\eps$, if the following holds. For any independent $(n,\kout)$-sources $X,Y$, there exist min-entropy $\kin$ subsources $X' \subset X$, $Y' \subset Y$, such that $\subext(X',Y')$ is $\eps$-close to uniform.
\end{definition}

Although we are not aware of the definition of two-source sub-extractors made explicit in previous works, we note that the two-source disperser constructed by Barak~\etal~\cite{BKSSW10} is in fact a two-source sub-extractor. More precisely, for any constant $\delta > 0$, the authors construct a two-source sub-extractor for outer-entropy $\delta n$ and inner-entropy $\poly(\delta)n$. On the other hand, the state of the art two-source disperser by Barak~\etal~\cite{brsw06} does not seem to be a sub-extractor.

The main theorem proved in this paper is the following.

\begin{theorem}[Two-source sub-extractors]\label{thm:subext main}
There exists an explicit two-source sub-extractor for outer-entropy $\kout = \polylog(n)$ and inner-entropy $\kin = \kout^{\Omega(1)}$, with $m = \kout^{\Omega(1)}$ output bits and error $\eps = 2^{-\kout^{\Omega(1)}}$.
\end{theorem}

We note that a sub-extractor for outer-entropy $\kout$ with $m$ output bits and error $\eps$ is a zero-error disperser for entropy $\kout$ with $\min(m, \log(1/\eps))$ output bits (the dependence in the inner-entropy $\kin$ is due to the fact that $m \le \kin$). Indeed, one can simply truncate the output of the sub-extractor to be short enough so that the error will be small enough to guarantee that any possible output is obtained. In particular, a sub-extractor for outer-entropy $\kout$ and inner-entropy $\kin = 1$, with error $\eps < 1/2$, induces a bipartite $2^{\kout}$-Ramsey-graph. Thus, Theorem~\ref{thm:subext main} readily implies Theorem~\ref{thm:main intro} and Theorem~\ref{thm:main intro dispersers}.

We further remark that the constants in Theorem~\ref{thm:subext main} depend on each other and so we made no attempt to optimize them. It is worth mentioning though that one can take $\kin = \kout^{1-\delta}$ for any constant $\delta > 0$, and even $\kin = \kout/\polylog(n)$ and still supporting outer-entropy $\kout = \polylog(n)$.

We hope that two-source sub-extractors can be of use in some cases where two-source extractors are required. Further, we believe that the techniques used for constructing sub-extractors are of value in future constructions of two-source extractors.

\subsection{Organization of this paper}

In Section~\ref{sec:cr overview} we give an informal overview of the challenge-response mechanism. Section~\ref{sec:overview} contains a comprehensive and detailed overview of our construction and analysis. These two sections are meant only for building up intuition. The reader may freely skip these sections at any point as we make no use of the results that appear in them.

In Section~\ref{sec:prelim} we give some preliminary definitions and results that we need. Section~\ref{sec:cr mec} contains the formal description of the challenge-response mechanism. In Section~\ref{sec:entropy trees} we present the notions of entropy-trees and tree-structured sources. Then, in Section~\ref{sec:the construction} we give the formal construction of our sub-extractor, and analyze it in Section~\ref{sec:anal}. Finally, in Section~\ref{sec:open} we list some open problems.

\section{Overview of the Challenge-Response Mechanism}\label{sec:cr overview}

Our construction of sub-extractors is based on the challenge-response mechanism that was introduced in~\cite{BKSSW10} and refined by~\cite{brsw06}. As we are aiming for a self-contained paper, in this section we explain how this powerful mechanism works. The challenge-response mechanism is delicate and fairly challenging to grasp. Thus, to illustrate the way the mechanism works, we give a toy example in Section~\ref{sec:overview playing with cr}.

\subsection{Motivating the challenge-response mechanism}\label{sec:overview for point six}

We start by recalling the notation of a block-source.

\begin{definition}
A random variable $X$ on $n$-bit strings is called an \emph{$(n,k)$-block-source}, or simply a $k$-block-source, if the following holds:
\begin{itemize}
  \item $\me(\lefty(X)) \ge k$, where $\lefty(X)$ is the length $n/2$ prefix of $X$.
  \item For any $x \in \sup(\lefty(X))$ it holds that $\me(\righty(X) \mid \lefty(X) = x) \ge k$, where $\righty(X)$ is the length $n/2$ suffix of $X$.
\end{itemize}
\end{definition}

In a recent breakthrough, Li~\cite{Li15} gave a construction of an extractor $\bext$ for two $n$-bit sources, where the first source is a $\polylog(n)$-block-source and the second is a weak-source with min-entropy $\polylog(n)$ (see Theorem~\ref{thm:li extractor}). Since our goal is to construct a two-source sub-extractor for outer-entropy $\polylog(n)$, a first attempt would be to show that any source $X$ with entropy $\polylog(n)$ has a subsource $X'$ that is a $\polylog(n)$-block-source. If this assertion were to be true then $\bext$ would have been a two-source sub-extractor.

This however is clearly not the case. Consider, for example, a source $X$ that all of its entropy is concentrated in its right-block $\righty(X)$. Namely, $\lefty(X)$ is fixed to some constant and $\righty(X)$ has min-entropy $k$. Clearly, $\me(X) \ge k$, yet no subsource of $X$ is even a $1$-block-source.

This example holds only when the entropy is no larger than $n/2$. Indeed, one cannot squeeze, say, $0.6n$ entropy to the $n/2$ bits of $\righty(X)$. Restricting ourselves, for the moment, to the very high entropy regime, we ask whether this example is the only problematic example. In particular, is it true that any source with min-entropy $0.6n$ is a block-source? The answer to this question is still no. Nevertheless, one can show that a $0.6n$-weak-source on $n$-bits has a low-deficiency subsource that is a $0.1n$-block-source. This observation will be important for us later on.

Note that by this observation, $\bext$ is a sub-extractor for two sources with min-entropy $0.6n, \polylog(n)$. However, by the example above, $\bext$ by itself is not a sub-extractor for two sources with min-entropy less than $0.5n$. Nevertheless, as we will see, $\bext$ is a central component in our construction.

Going back to the example, if only there were a magical algorithm that given a single sample $x \sim X$, would have been able to determine correctly whether or not $\lefty(X)$ is fixed to a constant, then we would have been in a better shape as we would have known to concentrate our efforts on $\righty(X)$. Such an algorithm, however, is too much to hope for. Indeed, given just one sample $x \sim X$, one simply cannot tell whether the left block of $X$ is fixed or not. Still, the powerful challenge-response mechanism allows one to accomplish almost this task. In the next section we present a slightly informal version of the challenge-response mechanism. The actual mechanism is described in Section~\ref{sec:cr mec}. Our presentation is somewhat more abstract than the one used in~\cite{BKSSW10,brsw06}.

\subsection{The challenge-response mechanism}\label{sec:overview cr mech}

We start by presenting a dream-version of the challenge-response mechanism.

\subsubsection*{The challenge-response mechanism -- dream version}
For integers $\ell < n$, a dream version of the \emph{challenge-response mechanism} would be a $\poly(n)$-time computable function
$$
\dreamresp \colon \{0,1\}^n \times \{0,1\}^n \times \{0,1\}^\ell \to \{ \fixed, \hasent \}
$$
with the following property. For any two independent $(n,\polylog(n))$-sources $X,Y$, and for any function $\ch \colon \{0,1\}^n \times \{0,1\}^n \to \{0,1\}^\ell$, the following holds:
\begin{itemize}
  \item If $\ch(X,Y)$ is fixed to a constant then
      $$
      \pr_{(x,y) \sim (X,Y)}\left[ \dreamresp\left(x,y,\ch(x,y)\right) = \fixed \right] = 1.
      $$
  \item If $\me(\ch(X,Y))$ is sufficiently large then
      $$
      \pr_{(x,y) \sim (X,Y)}\left[ \dreamresp\left(x,y,\ch(x,y)\right) = \hasent \right] = 1.
      $$
\end{itemize}

Note that for any function $\ch$, $\dreamresp$ distinguishes between the case that $\ch(X,Y)$ is fixed and the case that $\ch(X,Y)$ has enough entropy. Unfortunately, $\dreamresp$ will remain a dream. The actual challenge-response mechanism requires more from the inputs and has a weaker guarantee on the output. The difference between the dream version and the actual challenge-response mechanism contributes to why our sub-extractor is defined the way it does, and so already in this section we present the actual challenge-response mechanism (in a slightly informal manner).

\subsubsection*{The actual challenge-response mechanism}

For integers $k < \ell < n$, the challenge-response mechanism is a $\poly(n)$-time computable function
$$
\crm \colon \{0,1\}^n \times \{0,1\}^n \times \{0,1\}^\ell \to \{ \fixed, \hasent \}
$$
with the following property. For any two independent $(n,\polylog(n))$-sources $X,Y$, and for any function $\ch \colon \{0,1\}^n \times \{0,1\}^n \to \{0,1\}^\ell$, the following holds:
\begin{itemize}
  \item If $\ch(X,Y)$ is fixed to a constant then there exist deficiency $\ell$ subsources $X' \subset X$, $Y' \subset Y$, such that
      $$
      \pr_{(x,y) \sim (X',Y')}\left[ \crm(x,y,\ch(x,y)) = \fixed \right] = 1.
      $$
  \item If for any deficiency $\ell$ subsources $\hat{X} \subset X$, $\hat{Y} \subset Y$ it holds that $\me(\ch(\hat{X},\hat{Y})) \ge k$, then
      $$
      \pr_{(x,y) \sim (X,Y)}\left[ \crm(x,y,\ch(x,y)) = \hasent \right] \ge 1 - 2^{-k}.
      $$
\end{itemize}

We emphasize the differences between the dream-version and the actual challenge-response mechanism. First, even if $\ch(X,Y)$ is fixed to a constant, it is not guaranteed that $\crm$ will correctly identify this on any sample from $(X,Y)$. In fact, it is not even guaranteed that $\crm$ will identify this correctly with high probability over the sample. The actual guarantee is that there exist low-deficiency subsources $X' \subset X$, $Y' \subset Y$, such that on any sample $(x,y) \sim (X',Y')$, $\crm$ will correctly output $\fixed$. As our goal is to construct a sub-extractor, this is good enough for us, as we can ``imagine'' that we are given samples from $X',Y'$ rather than from $X,Y$ for the rest of the analysis (we do have to be careful when dealing with error terms when moving to subsources, but we will ignore this issue for now).

The second thing to notice is that for the challenge-response mechanism to identify the fact that $\ch(X,Y)$ has entropy, a stronger assumption is made. Namely, it is not enough that $\ch(X,Y)$ has a sufficient amount of entropy, but rather we need that $\ch(\hat{X},\hat{Y})$ has enough entropy for \emph{all} low-deficiency subsources $\hat{X} \subset X$, $\hat{Y} \subset Y$. So, informally speaking, for the challenge-response mechanism to identify entropy, this entropy must be robust in the sense that the entropy exists even in all low-deficiency subsources. Further, note that unlike the first case, in the second case $\crm$ introduces a small error.

\subsection{The three-types lemma}\label{sec:overview three cases lemma}

The challenge-response mechanism is indeed very impressive. However, the mechanism only distinguishes between two extreme cases -- no entropy versus high entropy. It is much more desired to be able to distinguish between low entropy versus high entropy. Indeed, what if the entropy in the left block of a source is too low to work with, yet the block is not constant and so the challenge-response mechanism is inapplicable?

The next lemma shows that if we are willing to work with subsources then this is a non-issue. Namely, every source has a low-deficiency subsource with a structure suitable for the challenge-response mechanism. We present here a slightly informal version of this lemma. The reader is referred to Lemma~\ref{lemma:three cases lemma} for a formal statement.

\begin{lemma}[The three-types lemma]\label{lemma:three cases lemma informal}
For any $(n,k)$-source $X$ and an integer $b < k$, there exists a deficiency $1$ subsource $X' \subset X$ such that (at least) one of the following holds:
\begin{itemize}
  \item $X'$ is a $b$-block-source.
  \item $\me(\lefty(X')) \ge k-b$.
  \item $\lefty(X')$ is fixed to a constant and $\me(\righty(X')) \ge k - b$.
\end{itemize}
\end{lemma}

Take for example $b = \sqrt{k}$. Lemma~\ref{lemma:three cases lemma informal}, which is a variant of the two-types lemma by Barak~\etal~\cite{brsw06}, states a fairly surprising fact about weak-sources. Any source $X$ has a deficiency $1$ subsource $X'$ with a useful structure. If $X'$ is not a block-source then either essentially all of the entropy already appears in $\lefty(X')$, or otherwise $\lefty(X')$ is fixed to a constant (great news for users of the challenge-response mechanism) and $\righty(X')$ has essentially all the entropy of $X$.

\subsection{Playing with the challenge-response mechanism}\label{sec:overview playing with cr}

Lemma~\ref{lemma:three cases lemma informal} is an important supplement to the challenge-response mechanism. However, it is still not even clear how the two together can be used to break the ``0.5 barrier'' discussed in Section~\ref{sec:overview for point six}. For example, how they together can be used to give a sub-extractor for outer-entropies $0.4n, \polylog(n)$.

Lets try to see what can be said. Say $X$ is an $(n,0.4n)$-source. By Lemma~\ref{lemma:three cases lemma informal}, applied with $b = \sqrt{n}$, there exists a deficiency $1$ subsource $X'$ of $X$, such that one of the following holds:
\begin{itemize}
  \item $X'$ is a $\sqrt{n}$-block-source.
  \item $\me(\lefty(X')) \ge 0.4n-\sqrt{n} \ge 0.3n$. 
  \item $\lefty(X')$ is fixed to a constant and $\me(\righty(X')) \ge 0.4n - \sqrt{n} \ge 0.3n$.
\end{itemize}

Note that in the second case, $\lefty(X')$ has entropy-rate $0.6$. Thus, it has a deficiency $1$ subsource that is a block-source. Similarly, in the third case, $\righty(X')$ has a deficiency $1$ subsource that is a block-source. Thus, any $(n,0.4n)$-source has a deficiency $2$ subsource $X'' \subset X$ such that at least one of $X''$, $\lefty(X'')$, $\righty(X'')$ is a $\sqrt{n}$-block-source.

Given this, even without resorting to the challenge-response mechanism, we know that at least one of $\bext(X'',Y), \bext(\lefty(X''),Y), \bext(\righty(X''),Y)$ is close to uniform. The challenge-response mechanism allows us to obtain something stronger. Although we will not be able to get a sub-extractor for outer-entropies $0.4n, \polylog(n)$ this way, it is instructive to see the technique being used. Set $\bext$ to output $\ell = o(k)$ bits, where $k = \polylog(n)$ is the outer-entropy of the second source. Consider the following algorithm.

\subsubsection*{The toy algorithm.}
On input $x,y \in \{0,1\}^n$
\begin{itemize}
  \item Compute $z(x,y) = \crm(x,y,\bext(\lefty(x),y))$.
  \item If $z = \fixed$ declare that $\bext(\righty(x),y)$ is uniform.
  \item Otherwise, declare that one of $\bext(x,y)$, $\bext(\lefty(x),y)$ is uniform.
\end{itemize}

The above algorithm does not look very impressive. Essentially, it only cuts down our lack of knowledge a bit. Instead of declaring that one of three strings is close to uniform, it is able to declare that one of at most two strings is close to uniform. Nevertheless, as mentioned above, it is instructive to see the proof technique on this simple toy example. Moreover, as we will see in Section~\ref{sec:overview brsw}, this algorithm is a special case of an algorithm by~\cite{brsw06} that will be important to our construction. We now prove that the algorithm's declaration is correct. More precisely, we prove the following.

\begin{claim}\label{claim:toy claim}
Let $X$ be an $(n,0.4n)$-source, and let $Y$ be an independent $(n,\polylog(n))$-source. Then, there exist deficiency $O(\ell)$-subsources $X' \subset X$, $Y' \subset Y$, such that with probability $1$ over $(x,y) \sim (X',Y')$ the declaration of the algorithm is correct.
\end{claim}

The proof of Claim~\ref{claim:toy claim} showcases the following three facts about low-deficiency subsources. Non of these facts is very surprising, but we make extensive use of them throughout the paper, and it is beneficial to see these facts in action on a simple example. Here we give slightly informal statements. For the formal statements see Fact~\ref{fact:def and entropy}, Fact~\ref{fact:def and fixing}, and Lemma~\ref{lem:subsource of block source}.

\begin{fact}\label{fact:def and entropy overview}
If $\me(X) \ge k$ and $X'$ is a deficiency $d$ subsource of $X$ then $\me(X') \ge k-d$.
\end{fact}

\begin{fact}\label{fact:def and fixing overview}
Let $X$ be a random variable on $n$-bit strings. Let $f \colon \{0,1\}^n \to \{0,1\}^\ell$ be an arbitrary function. Then, there exists $a \in \{0,1\}^\ell$ and a deficiency $\ell$ subsource $X'$ of $X$ such that $f(x) = a$ for every $x \in \sup(X')$.
\end{fact}

\begin{lemma}\label{lem:subsource of block source overview}
Let $X$ be a $k$-block-source, and let $X'$ be a deficiency $d$ subsource of $X$. Then, $X'$ is a $k-d$ block-source.
\end{lemma}

\begin{proof}
We start by applying the three-types lemma as discussed above so to obtain a deficiency $2$ subsource of $X$ with the properties lists above. For ease of notation, we denote this source by $X$.

Assume first that $\lefty(X)$ is fixed. Note that in this case $\bext(\lefty(X),Y)$ is a deterministic function of $Y$. Since the output length of $\bext$ is $\ell$, Fact~\ref{fact:def and fixing overview} implies that there exists a deficiency $\ell$ subsource $Y' \subset Y$ such that $\bext(\lefty(X),Y')$ is fixed to a constant. We are now in a position to apply the challenge-response mechanism so to conclude that there exist deficiency $\ell$ subsources $X' \subset X$, $Y'' \subset Y'$ such that
\begin{equation}\label{eq:first one}
\pr\left[z(X',Y'') = \fixed \right] = 1.
\end{equation}
Note that in this case, $\me(\righty(X)) \ge 0.3n$. Therefore, by Fact~\ref{fact:def and entropy overview}, $\me(\righty(X')) \ge 0.3n - \ell \ge 0.29n$. Since $\righty(X')$ has length $n/2$, we have that $\righty(X')$ has entropy-rate $0.58$, and so by Lemma~\ref{lemma:three cases lemma informal}, there exists a deficiency $1$ subsource $X''$ of $X'$ such that $X''$ is an $\Omega(n)$-block-source. Similarly, since $\me(Y) = k = \omega(\ell)$, Fact~\ref{fact:def and entropy overview} implies that $\me(Y'') = k - 2\ell = \polylog(n)$. Thus, $\bext(\righty(X''), Y'')$ is close to uniform.

To summarize, in the case that $\lefty(X)$ is fixed, there exist deficiency $O(\ell)$-subsources $X'' \subset X$, $Y'' \subset Y$ on which the algorithm correctly declares that $\bext(\righty(X''),Y'')$ is uniformly distributed.

Consider now the case that $\lefty(X)$ is not fixed. By Lemma~\ref{lemma:three cases lemma informal}, either $X$ is a $\sqrt{n}$-block-source, or otherwise $\me(\lefty(X)) \ge 0.3n$, and so $\lefty(X)$ has a deficiency $1$ subsource that is a $\sqrt{n}$-block-source. Therefore, the algorithm's declaration in this case is correct even on deficiency $O(1)$ subsources.
\end{proof}

In this section we gained some familiarity with the challenge-response mechanism and with the three-types lemma  (Lemma~\ref{lemma:three cases lemma informal}), which is an important supplement for the mechanism. Hopefully, this experience will assist the reader in the sequel.

\section{Overview of the Construction and Analysis}\label{sec:overview}

In this section, we present our construction of sub-extractors and give a comprehensive and detailed overview of the proof, though we allow ourselves to be somewhat imprecise whenever this contributes to the presentation. The formal proof, which can easily be recovered by the content of this section, appears in Section~\ref{sec:anal}. In Section~\ref{sec:overview entropy trees}, we introduce the notions of entropy-trees and tree-structured sources. A variant of this notion was used by~\cite{brsw06}. Then, in Section~\ref{sec:overview brsw}, we overview the approach taken by~\cite{brsw06} for their construction of two-source dispersers. Once the results needed from~\cite{brsw06} are in place, in Section~\ref{sec:overview our approach} we give an overview for the rest of our construction. In the following sections of this overview we give further details.

\subsection{Entropy-trees and tree-structured sources}\label{sec:overview entropy trees}

\subsubsection*{Motivating the notion of an entropy-tree}

We already saw that a source with entropy-rate $0.6$ has a deficiency $1$ subsource that is a block-source. By applying the three-types lemma (Lemma~\ref{lemma:three cases lemma informal}), we saw that any source $X$ with entropy-rate $0.4$ has a deficiency $2$ subsource that is either a block-source, or otherwise one of $\lefty(X), \righty(X)$ is a block-source. We, however, are interested in sources $X$ with only $\polylog(n)$ entropy. Is it true that there is a block-source ``lying somewhere'' in $X$ (or in a low-deficiency subsource of $X$)? Yes it is, though we have to dig deeper.

Say $X$ has min-entropy $k$. Lemma~\ref{lemma:three cases lemma informal}, set with $b = \sqrt{k}$, implies that there exists a deficiency $1$ subsource $X'$ of $X$ that is either a $\sqrt{k}$-block-source, or otherwise one of $\lefty(X'), \righty(X')$ has almost all the entropy of $X$. In other words, if $X'$ is not a block-source then the entropy-rate of one of $\lefty(X'), \righty(X')$ has almost doubled.

Assume that $X'$ is not a block-source, and that $\lefty(X')$ has entropy $k-\sqrt{k}$. By Lemma~\ref{lemma:three cases lemma informal}, set again with $b = \sqrt{k}$, there exists a deficiency $1$ subsource $X'' \subset X'$ such that either $\lefty(X'')$ is a $\sqrt{k}$-block-source, or otherwise one of $\lefty(\lefty(X''))$, $\righty(\lefty(X''))$ has min-entropy $k-2\sqrt{k}$. That is, if also $\lefty(X'')$ is not a $\sqrt{k}$-block-source then the entropy-rate of one of $\lefty(\lefty(X'')),\righty(\lefty(X''))$ is almost four-times the original entropy-rate of $X$.

At some point we are bound to find a block-source. Indeed, if we failed to find a block-source in the first $r$ iterations then there is a deficiency $r$ subsource $X^{(r)}$ of $X$ and a length $n \cdot 2^{-r}$ block of $X^{(r)}$ that has entropy $k-r\sqrt{k}$. Thus, if $k-r\sqrt{k} > 0.6 n \cdot 2^{-r}$ then this block has a deficiency $1$ subsource that is a block-source. In particular, if $k = \omega(\log^{2}{n})$ then a block-source will be found in the first $\log(n/k)+O(1)$ iterations.

As we apply Lemma~\ref{lemma:three cases lemma informal} at most $\log{n}$ times and since in each application we move to a deficiency $1$ subsource, we conclude that every $(n,k)$-source has a deficiency $\log{n}$ subsource that contains a block-source. This block-source can be found by following a certain ``path of entropy'' that determines which of the two halves of the current block of the source contains essentially all the entropy.

\subsubsection*{Entropy-trees}

The above discussion naturally leads to what we call an \emph{entropy-tree} and sources that have a tree-structure. An entropy-tree is a complete rooted binary tree $T$ where some of its nodes are labeled by one of the following labels: $\h,\bs,\f$, stand for high entropy, block-source, and fixed, respectively. The nodes of an entropy-tree are labeled according to rules that capture any possible entropy structure of a subsource obtained by the process described above. The rules are:

\begin{figure}[h]
\centering
\includegraphics[scale=0.5]{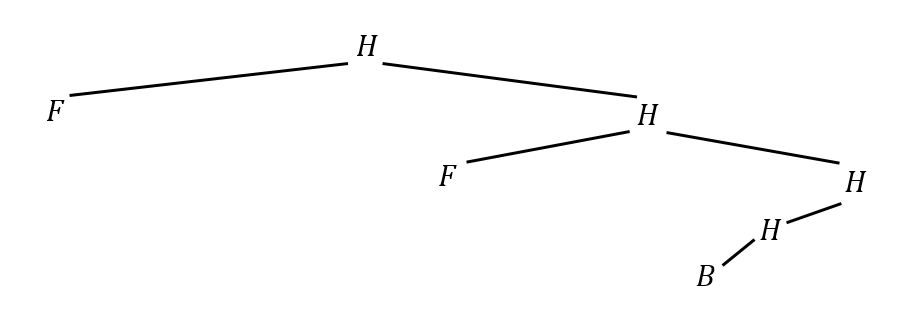}
\caption{An example of an entropy-tree. Unlabeled nodes and edges to them do not appear in the figure.}
\end{figure}

\begin{itemize}
  \item The root of $T$, denoted by $\root(T)$, is labeled by either $\h$ or $\bs$, expressing the fact that we assume the source itself has high entropy, and may even be a block-source.
  \item There is exactly one node in $T$ that is labeled by $\bs$, denoted by $v_\bs(T)$. This expresses the fact we proved, namely, if one digs deep enough, a block-source will be found. The uniqueness of the node labeled by $\bs$ captures the fact that we terminate the process once a block-source is found.
  \item If $v$ is a non-leaf that has no label, or otherwise labeled by $\f$ or $\bs$, then its sons have no label. This rule captures the fact that a node has no label if we are not interested in the block of the source that is associated with the node. Thus, if a block is fixed we do not try to look for a block-source inside it. Similarly, if the node is a block-source we stop the search.
  \item If $v$ is a non-leaf that is labeled by $\h$ then the sons of $v$ can only be labeled according to the following rules:
      \begin{itemize}
        \item If $\leftson(v)$ is labeled by $\f$ then $\rightson(v)$ is labeled by either $\h$ or $\bs$.
        \item If $\leftson(v)$ is labeled by either $\h$ or $\bs$ then $\rightson(v)$ has no label.
      \end{itemize}
      Note that these rules capture the guarantee of Lemma~\ref{lemma:three cases lemma informal}.
\end{itemize}

\paragraph{The entropy-path.} With every entropy-tree $T$ we associate a path that we call the \emph{entropy-path} of $T$. This is the unique path from $\root(T)$ to $v_\bs(T)$. We say that a path in $T$ contains the entropy-path if it starts at $\root(T)$ and goes through $v_\bs(T)$ (note that we allow an entropy-tree to have nodes that are descendants of $v_\bs(T)$. We just do not allow these nodes to be labeled).

\subsubsection*{Tree-structured sources}

Now that we have defined entropy-trees, we can say what does it mean for a source to have a $T$-structure, for some entropy-tree $T$. To this end we need to introduce some notations.
Let $n$ be an integer that is a power of $2$. With a string $x \in \{0,1\}^n$, we associate a depth $\log{n}$ complete rooted binary tree, where with each node $v$ of $T$ we associate a substring $x_v$ of $x$ in the following natural way. $x_{\root(T)} = x$, and for $v \neq \root(T)$, if $v$ is the left son of its parent, then $x_v = \lefty(x_{\parent(v)})$; otherwise, $x_v = \righty(x_{\parent(v)}$).

Let $T$ be a depth $\log{n}$ entropy-tree. An $n$-bit source $X$ is said to have a $T$-structure with parameter $k$ if for any node $v$ in $T$ the following holds:
\begin{itemize}
  \item If $v$ is labeled by $\f$ then $X_v$ is fixed to a constant.
  \item If $v$ is labeled by $\h$ then $\me(X_v) \ge k$.
  \item If $v$ is labeled by $\bs$ then $X_v$ is a $\sqrt{k}$-block-source.
\end{itemize}

With the notions of entropy-trees and tree-structured sources, we can summarize the discussion of this section by saying that any $(n,k)$-source, with $k = \omega(\log^{2}{n})$, has a deficiency $\log{n}$-subsource that has a $T$-structure with parameter $\Omega(k)$ for some entropy-tree $T$ (that depends on the underlying distribution of $X$). Therefore, for the purpose of constructing sub-extractors, we may assume that we are given two independent samples from tree-structured sources rather than from general weak-sources. Further, by Fact~\ref{fact:def and entropy overview} and by Lemma~\ref{lem:subsource of block source overview} it follows that if $X'$ is a deficiency $d$ subsource of a source having a $T$-structure with parameter $k = \omega(d)$, then $X'$ has a $T$-structure with parameter $\Omega(k)$. In particular, we can move to $o(k)$-deficiency subsources throughout the analysis and still maintain the original tree-structure of the source.

\subsection{Identifying the entropy-path}\label{sec:overview brsw}

Tree-structured sources certainly seem nicer to work with than general weak-sources. However, it is still not clear what good is this structure for if we do not have any information regarding the entropy-path.

Remarkably, by applying the challenge-response mechanism in a carefully chosen manner, Barak~\etal~\cite{brsw06} were able to identify the entropy-path of the entropy-tree $T$ given just one sample from $x \sim X$, where $X$ is a $T$-structured source, and one sample from $y \sim Y$, where $Y$ is a general weak-source that is independent of $X$. We now turn to describe the algorithm used by~\cite{brsw06}. Before we do so, it is worth mentioning that Barak~\etal proved something somewhat different. Indeed, they considered a variant of entropy-trees and had to prove something a bit stronger than what we need. Nevertheless, their proof can be adapted in a straightforward manner to obtain the result we describe next. For completeness, we reprove what is needed for our construction in Section~\ref{sec:anal of step 1}.

\subsubsection*{What does it mean to identify the entropy-path?}

What do we mean by saying that an algorithm identifies the entropy-path of an entropy-tree $T$? This is an algorithm that on input $x,y \in \{0,1\}^n$, outputs a depth $\log{n}$ rooted complete binary tree and a marked root-to-leaf path on that tree, denoted by $\pobs(x,y)$ -- the \emph{observed} entropy-path. Ideally, the guarantee of the algorithm would have been the following. If $x$ is sampled from a $T$-structured source $X$ and $y$ is sampled independently from a weak-source $Y$, then $\pobs(X,Y)$ contains the entropy-path of $T$ with probability $1$. That is, for any $(x,y) \in \sup((X,Y))$, if we draw the path $\pobs(x,y)$ on the entropy-tree $T$ then this path starts at $\root(T)$ and goes through $v_{\bs}(T)$.

Note that the path $\pobs(x,y)$ is allowed to continue arbitrarily after visiting $v_{\bs}(T)$. Requiring that $\pobs(x,y)$ will stop at $v_{\bs}(T)$ is a very strong requirement. In particular, it will conclude the construction of the sub-extractor. Indeed, once the block-source $X_{v_{\bs}(T)}$ is found, one can simply output $\bext(X_{v_{\bs}(T)}, Y)$.

This was an ideal version of what we mean by identifying an entropy-path. For our needs, we will be satisfied with a weaker guarantee. Following~\cite{brsw06}, we will show that there exist low-deficiency subsources $X' \subset X$, $Y' \subset Y$, such that with high probability over $(x,y) \sim (X',Y')$ it holds that $\pobs(x,y)$ contains the entropy-path of $T$.

The fact that we only have a guarantee on low-deficiency subsources is good enough for us as we are aiming for a sub-extractor. The fact that there is an error (that did not appear in the analysis of~\cite{brsw06}) should be handled with some care. Indeed, note that by moving to a deficiency $d$ subsource, an $\eps$ error in the original source can grow to at most $2^{d} \cdot \eps$ restricted to the subsource. We will make sure that the error is negligible compared to the deficiency we consider in the rest of the analysis. Thus, from here on we will forget about the error introduced in this step of identifying the entropy-path.

\subsubsection*{The algorithm of~\cite{brsw06} for identifying the entropy-path}

We now describe the algorithm used by~\cite{brsw06} for identifying the entropy-path of an entropy-tree $T$. The basic idea was depicted already in the toy algorithm from Section~\ref{sec:overview playing with cr}. In fact, what the toy algorithm was actually managed to do was to find the entropy-path in depth $2$ tree-structured sources.

We first note that if $\root(T) = v_{\bs}(T)$ then any observed entropy-path will contain $v_{\bs}(T)$. So, we may assume that this is not the case. Let $v$ be the parent of $v_{\bs}(T)$ in $T$. As a first step, we want to determine which of the two sons of $v$ is $v_{\bs}(T)$. To this end, we will use the toy algorithm from Section~\ref{sec:overview playing with cr}. More precisely, node $v$ declares that its left son is $v_{\bs}(T)$ if and only if
\begin{equation}\label{eq:who is my favored son}
\resp\left( x_{v}, y, \bext\left(x_{\leftson(v)}, y\right) \right) = \hasent.
\end{equation}
Lets pause for a moment to introduce some notations. If Equation~\eqref{eq:who is my favored son} holds, we say that the node $v$ $(x,y)$-favors its left-son; otherwise, we say that $v$ $(x,y)$-favors its right son. Moreover, we define the \emph{good son} of $v$ to be $v_{\bs}(T)$. More generally, for a node $u \neq v_{\bs}(T)$ that is an ancestor of $v_{\bs}(T)$, we define the \emph{good son} of $u$ to be its unique son that is an ancestor of $v_{\bs}(T)$. Note that by following the good sons from $\root(T)$ to $v_{\bs}(T)$ one recovers the entropy-path of $T$. Thus, to recover the entropy-path of $T$ it is enough that any ancestor of $v_{\bs}(T)$ on the entropy-path of $T$ favors its good son.

By following the proof of Claim~\ref{claim:toy claim}, one can see that if $X_{\leftson(v)}$ is fixed then Equation~\eqref{eq:who is my favored son} holds with probability $0$ on some low-deficiency subsources of $X,Y$. Further, by the challenge-response mechanism together with Fact~\ref{fact:def and entropy overview} and Lemma~\ref{lem:subsource of block source overview}, one can show that if $\leftson(v) = v_{\bs}(T)$ then with high probability over $(X,Y)$, Equation~\eqref{eq:who is my favored son} holds. Observe that by the definition of an entropy-tree, these are the only two possible cases.

We showed how $v_{\bs}(T)$ can convince its parent $v$ that it is its good son. The trick was to use the block-source-ness of $X_{v_{\bs}(T)}$ so to generate a proper challenge. Considering one step further, we ask the following. If $u$ is the parent of $v$, how can $v$ convince $u$ that it is its good son? After all, $v$ is not a block-source. The elegant solution of Barak~\etal is as follows. Given $x,y \in \{0,1\}^n$, the challenge of $v$ will contain not only $\bext(x_v, y)$ but also $\bext(x_w, y)$, where $w$ is $v$'s favored son. Thus, if $v$'s favored son happens to be its good son $v_{\bs}(T)$, then the challenge posed by $v$ will not be responded by $u$.

More generally, node $v$ decides which of its two sons it $(x,y$)-favors not according to Equation~\eqref{eq:who is my favored son} but rather according to wether or not
\begin{equation}\label{eq:who is my favored son really}
\resp\left( x_{v}, y, \ch\left(x_{\leftson(v)},y\right) \right) = \hasent,
\end{equation}
where $\ch(x_{\leftson(v)},y)$ is a matrix with at most $\log{n}$ rows (according to the depth of the tree) that contains $\bext(x_{\leftson(v)}, y)$ as row, as well as $\bext(x_w, y)$, where $w$ is the $(x,y)$-favored son of $\leftson(v)$, and also $\bext(x_r,y)$, where $r$ is the $(x,y)$-favored son of $w$, etc.

\subsubsection*{The strategy of~\cite{brsw06} for determining the output}

Having found the entropy-path of $T$, we are in a much better shape. We know that one of the nodes on the path is a block-source. The trouble is that we still do not know which one. We conclude this section by saying only a few words about the strategy taken by~\cite{brsw06} as at this point our strategy deviates from theirs. It is worth mentioning that the strategy taken by Barak~\etal for determining the output is one place in the construction that poses a bottleneck for supporting entropy $o(2^{\sqrt{\log{n}}})$. It is also the reason why the number of output bits in their construction is at most $O(\log\log{n})$ and why the construction can only be a disperser rather than a sub-extractor.

In order to output a non-constant bit, as required by a $1$ output bit disperser, Barak~\etal assumed that the source $X$ has some more structure. Not only $X$ should have a $T$-structure, but it is also required that $\lefty(X_{v_{\bs}(T)})$ has its own tree-structure. In particular, somewhere in the left block of $X_{v_{\bs}(T)}$ there should be a second block-source. Note that this extra structure required from $X$ can be assumed with almost no cost in parameters. Indeed, after applying the process from Section~\ref{sec:overview entropy trees} to $X$ so to obtain a deficiency $\log{n}$ subsource $X'$ of $X$ that has a tree-structure, one can simply apply the process again, this time to $\lefty(X'_{v_{\bs}(T)})$, so to find a deficiency $2\log{n}$ subsource of $X$ with the desired structure.

Having this ``double-block-source'' structure, Barak~\etal were able to carefully tune the parameters of the challenge-response mechanism so that with some probability $v_{\bs}(T)$ will be convinced that $X_{\leftson(v_{\bs}(T))}$ contains a block-source, yet with some probability it will fail to notice this. With some more delicate work, the fact that $v_{\bs}(T)$'s decision is not constant can be carried upwards all the way to $\root(T)$ and in turn, can be translated to an output bit that is non-constant.

\subsection{The strategy for the rest of our construction}\label{sec:overview our approach}

To carry the analysis of our sub-extractor, we require even more structure from our sources than the structure required by~\cite{brsw06}. First, we require \emph{both} $X$ and $Y$ to have a tree-structure. In previous works~\cite{BKSSW10,brsw06}, the second source $Y$ was used mainly to ``locate the entropy'' of the source $X$, and the only assumption on $Y$ was that it has a sufficient amount of entropy. We however will make use of the structure of $Y$ as well.

\begin{figure}[h]
\centering
\includegraphics[scale=0.5]{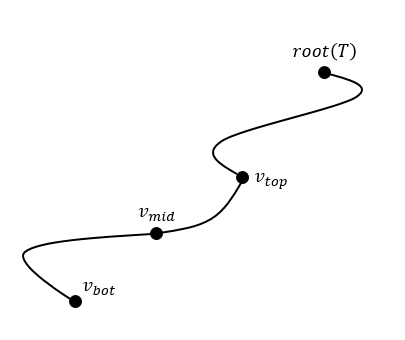}
\caption{The ``triple-block-source'' structure of an entropy-tree.}
\end{figure}

Second, we will need $X$ to have a ``triple-block-source'' structure. That is, we assume that $X$ has a $T_X$-structure with a node $\vvtop(T_X)$ corresponding to the block-source $X_{\vvtop(T_X)}$. We then assume that $\lefty(X_{\vvtop(T_X)})$ has its own tree-structure with a node $\vvmid(T_X)$ corresponding to a second block-source $X_{\vvmid(T_X)}$. Finally, we require that $\lefty(X_{\vvmid(T_X)})$ has its own tree-structure with a node $\vvbot(T_X)$ that corresponds to a third block-source $X_{\vvbot(T_X)}$.

As for $Y$, our analysis only requires a ``double-block-source'' structure. Though, to keep the notation cleaner we will assume that $Y$ also has a triple-block-source structure. In particular, the entropy-tree of $Y$, denoted by $T_Y$, has nodes that we denote by $\uutop(T_Y)$, $\uumid(T_Y)$, and $\uubot(T_Y)$ analogous to $\vvtop(T_X)$, $\vvmid(T_X)$, and $\vvbot(T_X)$ in $T_X$.

In fact, we allow ourselves to change the definition of an entropy-tree given in the previous section so that it will capture this ``triple-block-source'' structure, but the reader should not worry about these details at this point. For the formal definition see Section~\ref{sec:entropy trees}.

Given this structure of the sources, we are ready to give a high-level overview of our construction. In the subsequent sections of the overview, we give further details. Let $X$ be a $T_X$-structured source and let $Y$ be a $T_Y$-structured source, for some entropy-trees $T_X, T_Y$. At the first step, the sub-extractor identifies the entropy-path of $T_X$ and the entropy-path of $T_Y$ using the algorithm of~\cite{brsw06}. More precisely, given the samples $x,y$, we compute two paths denoted by \begin{align*}
\pobs(x,y) &= v_0(x,y), v_1(x,y), \ldots, v_{\log(n)-1}(x,y), \\
\qobs(x,y) &= u_0(x,y), u_1(x,y), \ldots, u_{\log(n)-1}(x,y).
\end{align*}
This step must be done with some care. From technical reasons, we cannot use $x,y$ to first find the entropy-path of $T_X$ and then to find the entropy-path of $T_Y$. Thus, in some sense, the two paths must be computed simultaneously.

At this point we have that there exist low-deficiency subsources $X' \subset X$, $Y' \subset Y$, such that for any $(x,y) \in \sup((X',Y'))$ it holds that $\pobs(x,y)$ (resp. $\qobs(x,y)$) contains the entropy-path of $T_X$ (resp. $T_Y$). In particular, we have that $v_{\dep(\vvtop(T_X))}(X',Y')$ is fixed to $\vvtop(T_X)$, and the same holds for $\vvmid(T_X), \vvbot(T_X)$, as well as for $\uutop(T_Y)$, $\uumid(T_Y)$, and $\uubot(T_Y)$. To keep the notation clean, we write $X,Y$ for $X',Y'$ in this  proof overview. That is, we assume that the entropy-paths are correctly identified on the tree-structured sources.

At the second step of the algorithm we identify $\vvmid(T_X)$ with high probability over subsources $X' \subset X$, $Y' \subset Y$. This sounds fantastic -- having found $\vvmid(T_X)$, we can simply output $\bext(X'_{\vvmid(T_X)}, Y')$. Unfortunately, however, the only way we know how to find $\vvmid(T_X)$ requires us to fix $\lefty(X'_{\vvmid(T_X)})$. That is, once found, $X'_{\vvmid(T_X)}$ is no longer a block-source. Moreover, to find $\vvmid(T_X)$ we also have to fix $\lefty(Y'_{\uumid(T_Y)})$.

We elaborate on how to find $\vvmid(T_X)$ in Section~\ref{sec:overview finding vmid}. Then in Section~\ref{sec:overview output}, we show how to determine the output of the sub-extractor even after loosing the block-structure of $X_{\vvmid(T_X)}$.

\subsection{Finding $\vvmid(T_X)$}\label{sec:overview finding vmid}
\subsubsection*{The node-path challenge and $\obsvvmid(x,y)$}
Given $x,y \in \{0,1\}^n$, the key idea we use for identifying $\vvmid(T_X)$ on $\pobs(x,y)$ lies in the design of a challenge that we call the \emph{node-path challenge}. Let $v$ be a node in $T_X$, and let $q = w_0, \ldots, w_{\log(n)-1}$ be a root-to-leaf path in $T_Y$. We define the challenge $\newch(x_v,y_q)$ to be the $\log(n)$-rows Boolean matrix such that for $i=0,1,\ldots,\log(n)-1$,
$$
\newch(x_v,y_q)_i = \bext\left(y_{w_i}, x_v\right).
$$
We define $\obsvvmid(x,y)$ to be the node $v$ on $\pobs(x,y)$ with the largest depth such that
\begin{equation}\label{eq:meshi ohelet}
\resp\left( x, y, \newch\left(x_v, y_{\qobs(x,y)}\right) \right) = \hasent.
\end{equation}

\begin{figure}[h]
\centering
\includegraphics[scale=0.6]{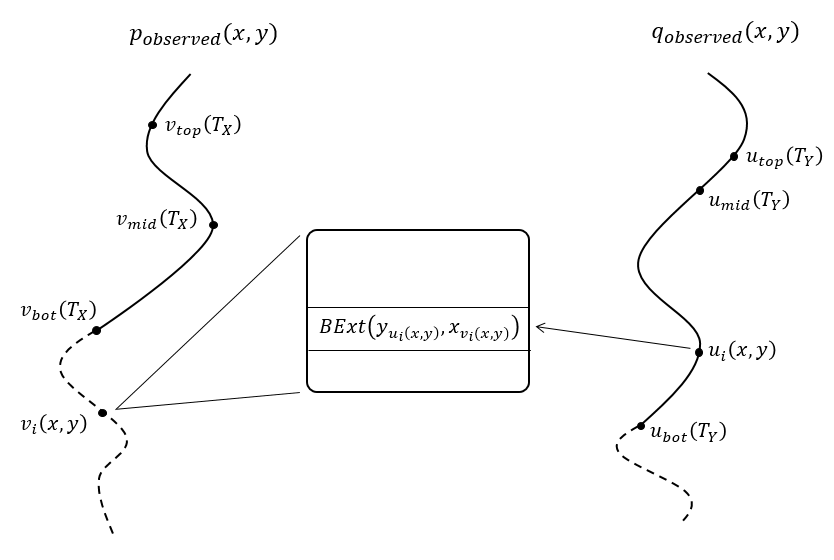}
\caption{The node-path challenge.}
\end{figure}

Ideally, we would want to prove that $\obsvvmid(x,y) = \vvmid(T_X)$ for any $(x,y) \in \sup((X,Y))$. By now we know that this is too much to ask for, and in any case, it suffices to prove that there exist low-deficiency subsources $X' \subset X$, $Y' \subset Y$ such that with high probability over $(x,y) \sim (X',Y')$ it holds that $\obsvvmid(x,y) = \vvmid(T_X)$. Unfortunately, we will not be able to prove that either. What we will be able to show is that there exist strings $\alpha,\beta$ such that the following holds. Define
\begin{align*}
X_\alpha &= X \mid \left( X_{\leftson(\vvmid(T_X))} = \alpha \right), \\
Y_\beta &= Y \mid \left( Y_{\leftson(\uumid(T_Y))} = \beta \right),
\end{align*}
and let $\imid(T_X)$ denote the depth of $\vvmid(T_X)$.

The way we choose $\alpha,\beta$ is with respect to the error that we constantly ignore throughout this overview. Thus, assume that $\alpha,\beta$ are chosen in such a way that allows us to continue ignoring the error. No further requirement is posed on $\alpha,\beta$.

\begin{proposition}\label{prop:heart overview}
There exist low-deficiency subsources $X_{\alpha,\beta} \subset X_\alpha$, $Y_{\alpha,\beta} \subset Y_\beta$, such that with high probability over $(x,y) \sim (X_{\alpha,\beta}, Y_{\alpha,\beta})$, it holds that
\begin{align*}
&\forall i > \imid(T_X) \quad \resp\left(x,y,\newch\left(x_{v_i(x,y)},y_{\qobs(x,y)}\right)\right) = \fixed,\\
&\resp\left(x,y,\newch\left(x_{\vvmid(T_X)},y_{\qobs(x,y)}\right)\right) = \hasent.
\end{align*}
\end{proposition}

Note that by the way we defined $\obsvvmid(x,y)$, Proposition~\ref{prop:heart overview} yields that $\obsvvmid(x,y) = \vvmid(T_X)$ with high probability over $(x,y) \sim (X_{\alpha,\beta}, Y_{\alpha,\beta})$. In particular, this gives us an algorithm for computing $\vvmid(T_X)$ -- simply go up the computed path $\pobs(x,y)$ until a node $v$ is found for which Equation~\eqref{eq:meshi ohelet} holds. In the rest of this section we prove Proposition~\ref{prop:heart overview}.

\subsubsection*{The challenges of descendants of $\vvmid(T_X)$ on $\pobs(x,y)$ are properly responded}

Proposition~\ref{prop:heart overview} has two parts. First, it states that the node-path challenges associated with nodes below $\vvmid(T_X)$ on the path $\pobs(x,y)$ are responded with high probability over the samples from some low-deficiency subsources of $X_\alpha,Y_\beta$. Second, Proposition~\ref{prop:heart overview} argues that the node-path challenge associated with $\vvmid(T_X)$ is unresponded with high probability over the samples.

Lets first consider the nodes below $\vvmid(T_X)$ on $\pobs(x,y)$. Naturally, we want to use the challenge-response mechanism. For that we must find low-deficiency subsources $X'_\alpha \subset X_\alpha$, $Y'_\beta \subset Y_\beta$ such that for all $i > \imid(T_X)$, the challenge
\begin{equation}\label{eq:newch overview}
\newch\left(
(X'_{\alpha})_{v_i(X'_{\alpha}, Y'_{\beta})}, (Y'_{\beta})_{\qobs(X'_\alpha,Y'_\beta)}\right)
\end{equation}
is fixed to a constant. As was done in the analysis of the toy algorithm from Section~\ref{sec:overview playing with cr}, to this end it is enough that the random variable
$$
\newch\left(
(X'_{\alpha})_{v_i(X'_{\alpha}, Y_{\beta})}, (Y_{\beta})_{\qobs(X'_\alpha,Y_\beta)}\right)
$$
is a deterministic function of $Y_\beta$. Indeed, in such case and since the challenge consists of a relatively small number of bits, we can apply Fact~\ref{fact:def and fixing overview} to find a low-deficiency subsource $Y'_{\beta} \subset Y_{\beta}$ such that Equation~\eqref{eq:newch overview} is fixed to a constant.

For $i > \imid(T_X)$, our starting point is the random variable
$$
\newch\left(
(X_{\alpha})_{v_i(X_{\alpha}, Y_{\beta})}, (Y_{\beta})_{\qobs(X_\alpha,Y_\beta)}\right).
$$
To make sure that this random variable depends solely on $Y_\beta$, we need to show that the dependence in all three syntactical appearances of $X_{\alpha}$ can be removed. We start with $\qobs(X_\alpha, Y_\beta)$.

\begin{claim}\label{claim:best claim ever}
There exists a deficiency $\log{n}$ subsource $X'_{\alpha}$ of $X_{\alpha}$ such that $\qobs(X'_\alpha, Y_\beta)$ is fixed to a constant.
\end{claim}

\begin{proof}
Let $\ibot(T_Y)$ denote the depth of $\uubot(T_Y)$. To prove the claim, we first recall that the path $\qobs(X_{\alpha},Y_{\beta})$ contains the entropy-path of $T_Y$. In particular, we have that the nodes $u_0(X_{\alpha},Y_{\beta}), \ldots, u_{\ibot(T_Y)}(X_{\alpha},Y_{\beta})$ are fixed. It is left to argue that there is a low-deficiency subsource $X'_\alpha \subset X_{\alpha}$ such that the remaining nodes $u_{\ibot(T_Y)+1}(X'_{\alpha}, Y_{\beta}), \ldots, u_{\log(n)-1}(X'_{\alpha}, Y_{\beta})$ are fixed as well.

Let us first consider the node $u_{\ibot(T_Y)+1}(X_{\alpha},Y_{\beta})$ that is the son of $u_{\ibot(T_Y)}(X_{\alpha},Y_{\beta}) = \uubot(T_Y)$. According to Equation~\eqref{eq:who is my favored son really}, node $\uubot(T_Y)$ decides which of its two sons will be on $\qobs(X_\alpha,Y_\beta)$ according to whether or not
\begin{equation}\label{eq:where the magic begins}
\resp\left( (Y_\beta)_{\uubot(T_Y)}, X_\alpha, \ch((Y_\beta)_{\leftson(\uubot(T_Y))},X_\alpha) \right) = \hasent.
\end{equation}
By the definition of an entropy-tree, $\uubot(T_Y)$ is a descendant of $\leftson(\uumid(T_Y))$. Further, by definition, $(Y_\beta)_{\leftson(\uumid(T_Y))}$ is fixed to $\beta$. Thus, also $(Y_\beta)_{\uubot(T_Y)}$ and $(Y_\beta)_{\leftson(\uubot(T_Y))}$ are fixed to some constants. Therefore, the Boolean expression in Equation~\eqref{eq:where the magic begins} is a deterministic function of $X_\alpha$. By applying Fact~\ref{fact:def and fixing overview}, we obtain a deficiency $1$ subsource $X'$ of $X_\alpha$ such that the Boolean expression in Equation~\eqref{eq:where the magic begins} is fixed. In particular, $u_{\ibot(T_Y)+1}(X',Y_{\beta})$ is fixed to a constant.

At this point we can apply the same argument to $\ibot(T_Y)+2$. Indeed, $u_{\ibot(T_Y)+1}(X',Y_{\beta})$ is fixed to a constant and all appearances of $Y_\beta$ in the Boolean expression that is analogous to Equation~\eqref{eq:where the magic begins} are again fixed to constants for the same reason as before. Since this process terminates after at most $\log{n}$ steps and since in each iteration we move to a deficiency $1$ subsource of the previous obtained subsource, the claim follows.
\end{proof}

Given Claim~\ref{claim:best claim ever}, we turn to show that for any $i > \imid(T_X)$,
\begin{equation}\label{eq:the newch overview}
\newch\left((X'_{\alpha})_{v_i(X'_{\alpha}, Y_{\beta})}, (Y_{\beta})_{\qobs(X'_{\alpha},Y_\beta)}\right)
\end{equation}
is a deterministic function of $Y_\beta$. By the discussion above, this will prove the first part of Proposition~\ref{prop:heart overview}.

By Claim~\ref{claim:best claim ever}, we already know that $\qobs(X'_\alpha, Y_\beta)$ is fixed to a constant. Thus, it suffices to show that $(X'_{\alpha})_{v_i(X'_{\alpha}, Y_{\beta})}$ is a deterministic function of $Y_\beta$ for all $i > \imid(T_X)$. By an argument similar to the one used in the proof of Claim~\ref{claim:best claim ever}, one can show that for any such $i$, $v_i(X'_\alpha,Y_\beta)$ is a deterministic function of $Y_\beta$. Note further that, by the definition of an entropy-tree, since $i > \imid(T_X)$ we have that $v_i(X'_\alpha,Y_\beta)$ is always (that is, for every $(x,y) \in \sup((X'_\alpha,Y_\beta))$) a descendant of $\leftson(\vvmid(T_X))$. Since $(X'_\alpha)_{\leftson(\vvmid(T_X))}$ is fixed to a constant we conclude that $(X'_{\alpha})_{v_i(X'_{\alpha},Y_{\beta})}$ is a deterministic function of $Y_\beta$.

By the discussion above, we are now in a position to apply Fact~\ref{fact:def and fixing overview} so to obtain a low-deficiency subsource $Y'_{\beta} \subset Y_{\beta}$ such that
$$
\newch\left((X'_{\alpha})_{v_i(X'_{\alpha}, Y'_{\beta})}, (Y'_{\beta})_{\qobs(X'_{\alpha},Y'_\beta)}\right)
$$
is fixed to a constant. We can then apply the challenge-response mechanism to show that there exist low-deficiency subsources $X_{\alpha,\beta} \subset X'_\alpha$, $Y_{\alpha,\beta} \subset Y'_\beta$ such that for any $(x,y) \in \sup((X_{\alpha,\beta}, Y_{\alpha,\beta}))$ it holds that
$$
\forall i > \imid(T_X) \quad \resp\left(x,y,\newch\left(x_{v_i(x,y)},y_{\qobs(x,y)}\right)\right) = \fixed.
$$

\subsubsection*{The challenge of $\vvmid(T_X)$ is unresponded}

To prove Proposition~\ref{prop:heart overview}, it is left to show that the node-path challenge associated with $\vvmid(T_X)$ is unresponded. More precisely, it suffices to show that with high probability over $(x,y) \sim (X_{\alpha,\beta}, Y_{\alpha,\beta})$, it holds that
$$
\resp\left(x,y,\newch\left(x_{\vvmid(T_X)},y_{\qobs(x,y)}\right)\right) = \hasent.
$$
Since $\uutop(T_Y)$ is on the path $\qobs(X_{\alpha,\beta},Y_{\alpha,\beta})$, the matrix
$$
\newch\left((X_{\alpha,\beta})_{\vvmid(T_X)},(Y_{\alpha,\beta})_{\qobs(X_{\alpha,\beta}, Y_{\alpha,\beta})}\right)
$$
contains the row
\begin{equation}\label{eq:bext rules}
\bext\left((Y_{\alpha,\beta})_{\uutop(T_Y)}, (X_{\alpha,\beta})_{\vvmid(T_X)}\right).
\end{equation}
Since $X_{\vvmid(T_X)}$ is a block-source, $(X_\alpha)_{\vvmid(T_X)}$ has a significant amount of entropy. Indeed, $X_\alpha$ is obtained from $X$ by fixing $X_{\leftson(\vvmid(T_X))} = \lefty(X_{\vvmid(T_X)})$. Therefore, by Fact~\ref{fact:def and fixing overview}, $(X_{\alpha,\beta})_{\vvmid(T_X)}$ also has a significant amount of entropy.

We now observe that $(Y_{\alpha,\beta})_{\uutop(T_Y)}$ is a block-source. Indeed, $Y_{\uutop(T_Y)}$ is a block-source and $Y_\beta$ is obtained from $Y$ by fixing $Y_{\leftson(\uumid(T_Y))}$. Since $Y_{\uumid(T_Y)}$ is a block-source, this fixing leaves some entropy in $(Y_\beta)_{\uumid(T_Y)}$. Recall further that $(Y_\beta)_{\uumid(T_Y)}$ lies inside $\lefty((Y_\beta)_{\uutop(T_Y)})$ as $\uumid(T_Y)$ is a descendant of $\leftson(\uutop(T_Y))$. Thus, we see that $(Y_{\alpha,\beta})_{\uutop(T_Y)}$ is a block-source.

Consider now any low-deficiency subsources $\hat{X} \subset X_{\alpha,\beta}$, $\hat{Y} \subset Y_{\alpha,\beta}$. By Fact~\ref{fact:def and fixing overview} and by Lemma~\ref{lem:subsource of block source overview} we have that $\hat{X}_{\vvmid(T_X)}$ has a significant amount of entropy and that $\hat{Y}_{\uutop(T_Y)}$ is a block-source (with some deterioration in parameters). Thus, for any low-deficiency subsources $\hat{X},\hat{Y}$ of $X_{\alpha,\beta}$, $Y_{\alpha,\beta}$, respectively, we have that the challenge matrix associated with $\vvmid(T_X)$ contains a row that is close to uniform. In particular this matrix is close to having high entropy. Thus, by the challenge-response mechanism, we have that the node-path challenge associated with $\vvmid(T_X)$ is unresponded with high probability over $(x,y) \sim (X_{\alpha,\beta}, Y_{\alpha,\beta})$, as desired.

\subsection{Determining the output}\label{sec:overview output}

At the last step of the algorithm we compute the output of the sub-extractor. The output of the sub-extractor is defined as
$$
\subext(x,y) = \bext\left(x_{\obsvvmid(x,y)} \circ x , y\right),
$$
where by $x_{\obsvvmid(x,y)} \circ x$ we denote the block-source with first block $x_{\obsvvmid(x,y)}$ and second block that equals $x$. Technically, we need to append the first block with zeros so that both blocks will have the same length, and also append $y$ with zeros.

There are two potential problems with applying $\bext$ the way we do above. First, we see that the block-source fed to $\bext$ depends on the sample $y$, which is problematic since $y$ is used as a sample from the weak-source as well. This, however, is a non-issue. Indeed, recall that with high probability over $(x,y) \sim (X_{\alpha,\beta}, Y_{\alpha,\beta})$ it holds that $\obsvvmid(x,y) = \vvmid(T_X)$, and so ignoring a small error, the computation of the extractor $\bext$ above is the same as
$$
\bext\left(x_{\vvmid(T_X)} \circ x, y\right).
$$
Now that we have shown that there are no dependencies between the two samples fed to $\bext$, we only need to make sure that the first sample is indeed coming from a block-source when sampling $(x,y) \sim (X_{\alpha,\beta}, Y_{\alpha,\beta})$.

Too see why this is true, recall that $\vvmid(T_X)$ is a descendant of $\leftson(\vvtop(T_X))$ and that $X_{\vvtop(T_X)}$ is a block-source. Since $X_{\alpha,\beta}$ is obtained from $X$ by fixing $Xֹ_{\leftson(\vvmid(T_X))}$ (and by moving to low-deficiency subsources) and since $X_{\vvmid(T_X)}$ is a block-source, we have that $(X_{\alpha,\beta})_{\vvtop(T_X)}$ is also a block-source. Therefore, $(X_{\alpha,\beta})_{\vvmid(T_X)} \circ X_{\alpha,\beta}$ is also a block-source. This shows that the application of $\bext$ above is valid, and that the output is close to uniform with high probability over $(X_{\alpha,\beta}, Y_{\alpha,\beta})$.

\section{Preliminaries}\label{sec:prelim}

\subsection{Standard (and less standard) notations and definitions}\label{sec:standard notations}

The logarithm in this paper is always taken base $2$. For every natural number $n \ge 1$, define $[n] = \{1,2,\ldots,n\}$.

\paragraph{Strings and matrices.} Let $n$ be an integer that is a power of $2$, namely $n = 2^m$ for some non-negative integer $m$. Let $x \in \{0,1\}^n$. For $i \in [n]$, we let $x_i$ denote the $i$'th bit of $x$. For $\emptyset \neq I \subseteq [n]$, we let $X_I$ denote the projection of $X$ to the coordinates in $I$. We denote by $\lefty(x)$ the $n/2$ leftmost bits of $x$ and by $\righty(x)$ the $n/2$ rightmost bits of $x$. That is, $\lefty(x) = x_1 \cdots x_{n/2}$ and $\righty(x) = x_{(n/2)+1} \cdots x_n$. We denote the concatenation of two strings $x,y$ by $x \circ y$. Given an $r \times n$ matrix $x$, for $i = 0,1,\ldots,r-1$, we let $x_i$ denote row $i$ of $x$.

\paragraph{Trees.} Let $T$ be a complete rooted binary tree. We denote the root of $T$ by $\root(T)$. Throughout the paper we consider trees where some of the nodes are labeled by labels from a ground set $L$. If $v$ is a labeled node in a tree $T$, we denote its label by $\lbl(v)$. If $v$ is a non-leaf in $T$, we denote the left and right sons of $v$ by $\leftson(v), \rightson(v)$, respectively. If $v$ is not the root of $T$, $\parent(v)$ denotes the parent of $v$. The depth of $T$ is denoted by $\dep(T)$. The depth of a node $v$ in $T$, denoted by $\dep(v)$, is the distance in edges from $\root(T)$ to $v$. Note that $\dep(\root(T)) = 0$.

\paragraph{Random variables and distributions.}
We sometimes abuse notation and syntactically treat random variables and their distribution as equal. Let $X,Y$ be two random variables. We say that $Y$ is a \emph{deterministic function of $X$} if the value of $X$ determines the value of $Y$. Namely, there exists a function $f$ such that $Y = f(X)$.

\paragraph{Associating strings with trees.}
Let $n$ be a power of $2$ and let $x \in \{0,1\}^n$. The tree that is associated with $x$, denoted by $T_x$, is a depth $\log{n}$ complete rooted binary tree, where with each node $v$ of $T_x$ we associate a substring $x_v$ of $x$ as follows:
\begin{itemize}
  \item $x_{\root(T)} = x$.
  \item For $v \neq \root(T)$, if $v$ is the left son of its parent, then $x_v = \lefty(x_{\parent(v)})$; otherwise, $x_v = \righty(x_{\parent(v)}$).
\end{itemize}

\paragraph{Statistical distance.}
The \emph{statistical distance} between two distributions $X,Y$ on a common domain $D$ is defined by
$$
\SD\left(X,Y\right) = \max_{A \subseteq D}\left\{\left|\, \Pr[X \in A] - \Pr[Y \in A]\,\right| \right\}.
$$
If $\SD(X,Y) \le \eps$ we say that $X$ is $\eps$-close to $Y$.

\paragraph{Min-entropy.}
The \emph{min-entropy} of a random variable $X$ is defined by
$$
\me(X) = \min_{x \in \sup(X)}{\log_2\left(\frac1{\Pr[X=x]}\right)}.
$$
If $X$ is supported on $\{0,1\}^n$, we define the \emph{min-entropy rate} of $X$ by $\me(X)/n$. In such case, if $X$ has min-entropy $k$ or more, we say that $X$ is an $(n,k)$-weak-source.


\subsection{Li's block-source--weak-source extractor}

Let $X$ be a random variable on $n$ bit strings, and assume $n$ is even. We say that $X$ is an \emph{$(n,k)$-block-source} if the following holds:
\begin{itemize}
  \item $\me(\lefty(X)) \ge k$.
  \item For any $x \in \sup(\lefty(X))$ it holds that $\me(\righty(X) \mid \lefty(X) = x) \ge k$.
\end{itemize}
We sometimes omit the length $n$ of $X$ and say that $X$ is a $k$-block-source.

In a recent breakthrough, Li~\cite{Li15} gave a construction of an extractor for two $n$-bit sources, where the first source is a $\polylog(n)$-block-source and the second is a weak-source with min-entropy $\polylog(n)$. Our construction heavily relies on Li's extractor.

\begin{theorem}[\cite{Li15}]\label{thm:li extractor}
There exists a universal constant $\gamma > 0$ such that the following holds. For all integers $n,k$ with $k \ge \log^{12}{n}$, there is a $\poly(n)$-time computable function
$$
\bext \colon \{0,1\}^{n} \times \{0,1\}^n \to \{0,1\}^m
$$
such that if $X$ is a $k$-block-source, where each block is on $n/2$ bits, and $Y$ is an independent $(n,k)$-source, then
$$
\SD\left( (\bext(X,Y),Y), (U_m, Y) \right) \le \eps,
$$
and
$$
\SD\left( (\bext(X,Y),X), (U_m, X) \right) \le \eps,
$$
where $m = 0.9k$ and $\eps = 2^{-k^{\gamma}}$.
\end{theorem}

Let $t<n$ be even integers. We sometimes apply $\bext$ on strings $x \in \{0,1\}^t$ and $y \in \{0,1\}^n$ and write $\bext(x,y)$. Formally, we actually compute $\bext(x',y)$ where $x'$ is obtained by appending $(n-t)/2$ zeros before and after $x$. This way of padding $x$ preserves the block-structure of $x$.

\subsection{Subsources}

The notion of a subsource was first explicitly introduced and studied by Barak~\etal~\cite{BKSSW10}. We start by giving the definition of a subsource and then collect some facts about subsources.

\begin{definition}[Subsource]
Given random variables $X$ and $X'$ on $\{0,1\}^n$, we say that $X'$ is a \emph{deficiency $d$} subsource of $X$ and write $X' \subset X$ if there exists a set $A \subseteq \{0,1\}^n$ such that $(X \mid A) = X'$ and $\pr[X \in A] \ge 2^{-d}$. More precisely, for every $a \in A$, $\pr[X' = a]$ is defined by $\pr[X = a \mid X \in A]$ and for $a \not\in A$, $\pr[X' = a] = 0$.
\end{definition}

\begin{fact}[\cite{brsw06}, Fact 3.11]\label{fact:def and entropy}
If $X$ is an $(n,k)$-source and $X'$ is a deficiency $d$ subsource of $X$ then $X'$ is an $(n,k-d)$-source.
\end{fact}

\begin{fact}[\cite{brsw06}, Fact 3.13]\label{fact:def and fixing}
Let $X$ be a random variable on $n$-bit strings. Let $f \colon \{0,1\}^n \to \{0,1\}^m$ be a function. Then, there exists $a \in \{0,1\}^m$ and a deficiency $m$ subsource $X'$ of $X$ such that $f(x) = a$ for every $x \in \sup(X')$.
\end{fact}

\begin{lemma}[\cite{brsw06}, Corollary 3.19]\label{lem:subsource of block source}
Let $X$ be a $k$-block-source, and let $X'$ be a deficiency $d$ subsource of $X$. Then, $X'$ is $\eps$-close to being a $k-d-\log(1/\eps)-1$ block-source.
\end{lemma}

\section{The Challenge-Response Mechanism}\label{sec:cr mec}

In this section we further abstract the challenge-response mechanism that was introduced in~\cite{BKSSW10} and refined by~\cite{brsw06}. This abstraction will make it easier for us to apply the mechanism in our proofs. The reader is referred to Section~\ref{sec:cr overview} for an intuitive-level overview of the mechanism.

\begin{theorem}\label{thm:cr}
For integers $\ell < n$, there exists a $\poly(n)$-time computable function
$$
\crm \colon \{0,1\}^n \times \{0,1\}^n \times \{0,1\}^\ell \to \{ \fixed, \hasent \}
$$
with the following property. For any two independent $(n,k)$-sources $X,Y$ with $k > \log^{10}{n}$ and for any function $\ch \colon \{0,1\}^n \times \{0,1\}^n \to \{0,1\}^\ell$, the following holds:
\begin{itemize}
  \item If $\ch(X,Y)$ is fixed to a constant then there exist deficiency $2\ell$ subsources $X' \subset X$, $Y' \subset Y$, such that
      $$
      \pr_{(x,y) \sim (X',Y')}\left[ \crm\left(x,y,\ch(x,y)\right) = \fixed \right] = 1.
      $$
  \item If for any deficiency $20\ell$ subsources $\hat{X} \subset X$, $\hat{Y} \subset Y$ it holds that $\ch(\hat{X},\hat{Y})$ is $\eps$-close to having min-entropy $k$, then
      $$
      \pr_{(x,y) \sim (X,Y)}\left[ \crm\left(x,y,\ch(x,y)\right) = \fixed \right] \le (2^{-k} + \eps) \cdot \poly(n).
      $$
\end{itemize}
\end{theorem}

The proof of Theorem~\ref{thm:cr} readily follows from the following theorem.

\begin{theorem}[Theorem 4.3, \cite{brsw06}]\label{thm:two sources to srs}
There exist universal constants $\gamma,c$ such that for any integer $n$, there exists a $\poly(n)$-time computable function
$$
\se \colon \{0,1\}^n \times \{0,1\}^n \to \left( \{0,1\}^\ell \right)^r,
$$
with $\ell = \gamma k$ and $r = n^c$, such that the following holds. For any $(n,k)$-independent sources $X,Y$, with $k > \log^{10}{n}$, it holds that:
\begin{itemize}

  \item Let $a$ be any fixed $\ell$ bit string. Then, there are subsources $X_a \subset_{2\ell} X$, $Y_a \subset_{2\ell} Y$ and an index $i \in [r]$ such that $\pr\left[\se(X_a,Y_a)_i = a\right] = 1$.

  \item Given any particular row index $i \in [r]$, $(X,Y)$ is $2^{-10\ell}$-close to a convex combination of subsources such that for every $(\hat{X}, \hat{Y})$ in the combination,
      \begin{itemize}
        \item $\hat{X}$ is a deficiency $20\ell$ subsource of $X$.
        \item $\hat{Y}$ is a deficiency $20\ell$ subsource of $Y$.
        \item $\hat{X}, \hat{Y}$ are independent.
        \item $\se(\hat{X}, \hat{Y})_i$ is fixed to a constant.
      \end{itemize}
\end{itemize}
\end{theorem}

\begin{proof}[Proof of Theorem~\ref{thm:cr}]
We first describe the algorithm for computing the response function $\resp(x,y,\ch(x,y))$ on input $x,y \in \{0,1\}^n$. The algorithm computes $\se(x,y)$, where the output length of $\se$ is set to $\ell$. The algorithm then checks whether or not $\ch(x,y)$ appears as a row in $\se(x,y)$. If so the algorithm outputs $\fixed$; otherwise it outputs $\hasent$.

We turn to the analysis. Assume first that $\ch(X,Y)$ is fixed. By Theorem~\ref{thm:two sources to srs}, there exist deficiency $2\ell$ subsources $X' \subset X$, $Y' \subset Y$ and an index $i \in [r]$ such that with probability $1$ over $(x,y) \sim (X',Y')$ it holds that $\se(x,y) = \ch(x,y)$, thus proving the first part of the theorem.

Assume now that for any deficiency $20\ell$ subsources $\hat{X} \subset X$, $\hat{Y} \subset Y$, it holds that $\ch(\hat{X},\hat{Y})$ is $\eps$-close to having min-entropy $k$. Consider any fixed $i \in [r]$. By Theorem~\ref{thm:two sources to srs}, $(X,Y)$ is $2^{-10\ell}$-close to a convex combination of subsources such that every $(\hat{X},\hat{Y})$ in the combination has the four listed properties. Since $\ch(\hat{X},\hat{Y})$ is $\eps$-close to having min-entropy $k$, we have that
$$
\pr_{(x,y) \sim (\hat{X},\hat{Y})}\left[ \ch(x,y) = \se(x,y)_i \right] \le 2^{-k} + \eps.
$$
Accounting for the distance from $(X,Y)$ to the convex combination,
$$
\pr_{(x,y) \sim (X,Y)}\left[ \ch(x,y) = \se(x,y)_i \right] \le 2^{-k} + \eps + 2^{-10\ell}.
$$
Therefore, by the union bound over all $i \in [r]$,
$$
\pr_{(x,y) \sim (X,Y)}\left[ \exists i \in [r] \quad \ch(x,y) = \se(x,y)_i \right] \le (2^{-k} + \eps + 2^{-10\ell}) r.
$$
As $r = \poly(n)$ and since $k \le \ell \le 10\ell$, the proof follows.
\end{proof}


\section{Entropy-Trees and Tree-Structured Sources}\label{sec:entropy trees}

\begin{definition}
An \emph{entropy-tree} $T$ is a complete rooted binary tree where some of the nodes of the tree are labeled by one of the following labels: $\f, \h, \btop, \bmid, \bbot$, according to the following set of rules:
\begin{itemize}
  \item $\lbl(\root(T)) \in \{\h, \btop \}$.
  \item There is exactly one node in $T$ that is labeled by $\btop$, one node that is labeled by $\bmid$, and one node labeled by $\bbot$, denoted by $\vvtop(T)$, $\vvmid(T)$ and $\vvbot(T)$, respectively. Further, $\vvmid(T)$ is a descendant of $\leftson(\vvtop(T))$, and $\vvbot(T)$ is a descendant of $\leftson(\vvmid(T))$. We denote $\itop(T) = \dep(\vvtop(T))$, $\imid(T) = \dep(\vvmid(T))$, and $\ibot(T) = \dep(\vvbot(T))$.
  \item If $v$ is a non-leaf that has no label or otherwise is labeled by $\f$ or $\bbot$ then both its sons have no label.
  \item If $v$ is a non-leaf labeled by $\h, \btop$ or $\bmid$ then $\leftson(v)$ has a label. Further,
  \begin{itemize}
    \item If $\lbl(\leftson(v)) = \f$ then $\rightson(v)$ has a label and $\lbl(\rightson(v)) \neq \f$.
    \item If $\lbl(\leftson(v)) \neq \f $ then the right son of $v$ has no label.
  \end{itemize}
\end{itemize}
\end{definition}

In our proofs we consider two sources, each having its own tree-structure. In such cases, we use $v$ to denote a node in one entropy-tree and $u$ to denote a node in the other entropy-tree. So, for example, we will use $\vvtop(T_X)$ to denote the node in the first entropy-tree labeled by $\btop$, whereas we use $\uutop(T_Y)$ to denote the node in the second entropy-tree labeled by $\btop$.

\begin{definition}[Entropy-path]
Let $T$ be an entropy-tree. The \emph{entropy-path} of $T$ is the path that starts at $\root(T)$ and ends at $\vvbot(T)$. We denote the nodes on this path by $\root(T) = v_0(T), \ldots, v_{\ibot(T)}(T) = \vvbot(T)$. We say that a path $p$ in $T$ contains the entropy-path of $T$ if $p$ starts at $\root(T)$ and goes through $\vvbot(T)$.
\end{definition}

\begin{definition}[Good son]
Let $T$ be an entropy-tree and let $v \neq \vvbot(T)$ be an ancestor of $\vvbot(T)$. The \emph{good son} of $v$ is defined to be the unique son of $v$ that is an ancestor of $\vvbot(T)$.
\end{definition}

\begin{definition}
Let $T$ be an entropy-tree. We say that an $n$-bit random variable $X$ has a \emph{$T$-structure} with parameters $(k,\eps)$ if the following holds. For any node $v$ in $T$:
\begin{itemize}
  \item If $\lbl(v) = \f$ then $X_v$ is fixed to a constant.
  \item If $\lbl(v) = \btop$ then $X_v$ is $\eps$-close to a $k^{1/2}$-block-source.
  \item If $\lbl(v) = \bmid$ then $X_v$ is $\eps$-close to a $k^{1/4}$-block-source.
  \item If $\lbl(v) = \bbot$ then $X_v$ is $\eps$-close to a $k^{1/8}$-block-source.
  \item If $\lbl(v) = \h$ then the following holds:
     \begin{itemize}
        \item If $v$ is an ancestor of $\vvtop(T)$ then $\me(X_v) \ge k$.
        \item If $v$ is a descendant of $\vvtop(T)$ and an ancestor of $\vvmid(T)$ then $\me(X_v) \ge k^{1/2}$.
        \item If $v$ is a descendant of $\vvmid(T)$ then $\me(X_v) \ge k^{1/4}$.
     \end{itemize}
\end{itemize}
\end{definition}

%

We further make use of the following lemma, which is analogous to the two-types lemma of Barak~\etal (see~\cite{brsw06}, Lemma 6.8).

\begin{lemma}[Three-types lemma]\label{lemma:three cases lemma}
For any $(n,k)$-source $X$ there exists a deficiency $1$ subsource $X' \subset X$ such that (at least) one of the following holds:
\begin{itemize}
  \item $\me(\lefty(X')) \ge k-\sqrt{k}-1$.
  \item $X'$ is a $\sqrt{k}$-block-source.
  \item $\lefty(X')$ is fixed to a constant and $\me(\righty(X')) \ge k - \sqrt{k} - 1$.
\end{itemize}
\end{lemma}

For the proof of Lemma~\ref{lemma:three cases lemma}, we make use of the following ``fixing entropies'' lemma by Barak~\etal~\cite{brsw06}. We state the lemma for a special case (and with slightly stronger parameters, which are easily achievable for that specific case by following the proof of~\cite{brsw06}).

\begin{lemma}[\cite{brsw06}, Lemma 3.20]\label{lem:brsw fixing entropies}
Let $X$ be an $(n,k)$-source. Let $0 < \tau_1 < \tau_2 < n$ be any two numbers. Set $\tau_0 = 0$ and $\tau_3 = n$. Then, there exist a deficiency $1$ subsource $X' \subset X$ and an index $i \in \{0,1,2\}$ such that the following holds:
\begin{itemize}
  \item For any fixing of $\lefty(X')$, $\me(\righty(X')) \in [\tau_i, \tau_{i+1}]$.
  \item $\me(\lefty(X')) + \tau_{i+1} \ge k-1$.
\end{itemize}
\end{lemma}

\begin{proof}[Proof of Lemma~\ref{lemma:three cases lemma}]
Set $\tau_1 = \sqrt{k}$, $\tau_2 = k - \sqrt{k} - 1$ and apply Lemma~\ref{lem:brsw fixing entropies} to obtain a deficiency $1$ subsource $X'' \subset X$ and $i \in \{0,1,2\}$. We consider three cases, according to the value of $i$.

\begin{itemize}
   \item If $i = 0$ then by the second item of Lemma~\ref{lem:brsw fixing entropies}, $\me(\lefty(X'')) + \tau_1 \ge k-1$. Thus, $\me(\lefty(X'')) \ge k-\sqrt{k}-1$. We then take $X' = X''$.

  \item If $i = 1$ then for any fixing of $\lefty(X'')$, we have that $\me(\righty(X'')) \in [\tau_1, \tau_2] = [\sqrt{k},k-\sqrt{k}-1]$. By the second item of Lemma~\ref{lem:brsw fixing entropies}, $\me(\lefty(X'')) \ge k - 1 - \tau_2 = \sqrt{k}$.
      Therefore, $X''$ is a $\sqrt{k}$-block-source, and we take $X' = X''$.
  \item If $i = 2$ then for any fixing of $\lefty(X'')$, we have that $\me(\righty(X'')) \ge \tau_2 = k - \sqrt{k} - 1$. We take $X'$ be a subsource of $X''$ conditioned on an arbitrary fixing of $\lefty(X'')$.
\end{itemize}

%
\end{proof}

The following fact follows by applying Lemma~\ref{lemma:three cases lemma} iteratively as was described in Section~\ref{sec:overview entropy trees} (see also Lemma 6.10 of~\cite{brsw06}, which proves essentially the same result).

\begin{fact}\label{prop:having lowdef t-source}
Let $X$ be an $(n,k)$-source with $k = \omega(\log^8{n})$. Then, there exist an entropy-tree $T$ and a deficiency $\log{n}$ subsource $X_T \subset X$ that has a $T$-structure with parameters $(k/2, 2^{-\Omega(k^{1/10})})$.
\end{fact}

\section{The Two-Source Sub-Extractor}\label{sec:the construction}

In this section we describe our two-source sub-extractor. Let $n$ be a power of $2$, and let $\ell$ be a parameter. On input $x,y \in \{0,1\}^n$, the computation of the sub-extractor is done in three steps.

\subsubsection*{Step 1 -- Identify the entropy-paths}

\paragraph{Setting the challenges.}With each node $v$ of $T_x$, we associate a $\log(n) \times \ell$ Boolean matrix, denoted by $\ch(x_v,y)$, computed from leaves to root, recursively as follows. All entries in rows $0,\ldots,\dep(v)-1$ of $\ch(x_v,y)$ are fixed to $0$. Row $\dep(v)$ of $\ch(x_v,y)$ contains $\bext(x_v,y)$, where $\bext$ is the extractor from Theorem~\ref{thm:li extractor} set to output $\ell$ bits. If $v$ is a non-leaf, rows $\dep(v)+1,\ldots,\log(n)-1$ are copied from the respective rows of $\ch(x_{\leftson(v)},y)$ or from the respective rows of $\ch(x_{\rightson(v)},y)$ according to the following rule. If
$$
\crm\left( x_v, y, \ch\left(x_{\leftson(v)},y\right) \right) = \fixed
$$
then the remaining rows are taken from the corresponding rows of $\ch(x_{\rightson(v)},y)$. Otherwise, the rows are taken from the corresponding rows of $\ch(x_{\leftson(v)},y)$. In the first case we say that $v$ $(x,y)$-favors its right son, and in the second case we say that $v$ $(x,y)$-favors its left son.

Analogously, with each node $u$ of $T_y$ we associate a $\log(n) \times \ell$ Boolean matrix, denoted by $\ch(y_u,x)$, defined recursively as follows. All entries in rows $0,\ldots,\dep(u)-1$ of $\ch(y_u,x)$ are fixed to $0$. Row $\dep(u)$ of $\ch(y_u,x)$ contains $\bext(y_u,x)$. If $u$ is a non-leaf, rows $\dep(u)+1,\ldots,\log(n)-1$ are copied from the respective rows of $\ch(y_{\leftson(u)},x)$ or from the respective rows of $\ch(y_{\rightson(u)},x)$ according to the following rule. If
$$
\crm\left( y_u, x, \ch\left(y_{\leftson(u)},x\right) \right) = \fixed
$$
then the remaining rows are taken from the corresponding rows of $\ch(y_{\rightson(u)},x)$. Otherwise, the rows are taken from the corresponding rows of $\ch(y_{\leftson(u)},x)$. In the first case we say that $u$ $(x,y)$-favors its right son, and in the second case we say that $u$ $(x,y)$-favors its left son.

\paragraph{Computing the entropy-paths.}

Let $v_0(x,y), v_1(x,y), \ldots, v_{\log(n)-1}(x,y)$ be the root-to-leaf path in $T_x$ defined by the property that $v_i(x,y)$ $(x,y)$-favors $v_{i+1}(x,y)$ for all $i=0,1,\ldots,\log(n)-2$. Similarly, let $u_0(x,y), u_1(x,y), \ldots, u_{\log(n)-1}(x,y)$ be the root-to-leaf path in $T_y$ defined by the property that $u_i(x,y)$ $(x,y)$-favors $u_{i+1}(x,y)$ for all $i=0,1,\ldots,\log(n)-2$. We denote $v_0(x,y), \ldots, v_{\log(n)-1}(x,y)$ by $\pobs(x,y)$ and call this path the \emph{observed entropy-path of $T_x$}. Similarly, we denote $u_0(x,y), \ldots, u_{\log(n)-1}(x,y)$ by $\qobs(x,y)$ and call this path the \emph{observed entropy-path of $T_y$}.

\paragraph{The computation done in Step 1.}
Given $x,y \in \{0,1\}^n$, at step 1 the sub-extractor computes $\pobs(x,y)$ and $\qobs(x,y)$. Clearly, this computation can be done in $\poly(n)$-time.

\subsubsection*{Step 2 -- Identify $\vvmid(T_X)$}

Given $x,y,\pobs(x,y)$, and $\qobs(x,y)$, at the second step the algorithm computes $\obsvvmid(x,y)$ as follows.

\paragraph{Setting the node-path challenges.}
Set $\ell' = \ell/\log^{3}{n}$. Let $v$ be a node in $T_x$ and let $p = w_0,\ldots,w_{\log(n)-1}$ be a root-to-leaf path in $T_y$. The node-path challenge associated with $(v,p)$, denoted by $\newch(x_v,y_p)$, is a $\log(n) \times \ell'$ Boolean matrix, defined as follows. For $j=0,\ldots,\log(n)-1$,
$$
\newch(x_v,y_p)_j = \bext(y_{w_j},x_v),
$$
where $\bext$ is the extractor from Theorem~\ref{thm:li extractor} set to output $\ell'$ bits.

\paragraph{Computing $\obsvvmid(x,y)$.}
We define $\obsvvmid(x,y)$ to be the node $v$ in $\pobs(x,y)$ with the largest depth such that
$$
\resp\left(x, y, \newch\left( x_v, y_{\qobs(x,y)} \right) \right) = \hasent.
$$
If no such node exists we define $v$, arbitrarily, as $\root(T_X)$. Note that computing $\obsvvmid(x,y)$ can be done in time $\poly(n)$.

\subsubsection*{Step 3 -- Determine the output}

Given $x,y$ and $\obsvvmid(x,y)$ computed in the previous step, he output of the sub-extractor is defined by
$$
\subext(x,y) = \bext\left(x_{\obsvvmid(x,y)} \circ x, y\right),
$$
where by $x_{\obsvvmid(x,y)} \circ x$ we mean the block-source with first block $x_{\obsvvmid(x,y)}$ and second block that contains $x$. Technically, we need to append $x_{\obsvvmid(x,y)}$ with zeros so to obtain a length $n$ string. Similarly, we append $y$ with $n$ zeros so to obtain a $2n$ bit string.

\subsubsection*{Recap}

We end this section by recapping the three steps in the computation of the sub-extractor. On input $x,y \in \{0,1\}^n$

\begin{enumerate}
  \item Compute $\pobs(x,y)$ and $\qobs(x,y)$.
  \item Compute $\obsvvmid(x,y)$.
  \item Output $\bext\left(x_{\obsvvmid(x,y)} \circ x, y\right)$.
\end{enumerate}


\section{Analysis of the Construction}\label{sec:anal}

In this section we prove Theorem~\ref{thm:subext main} by analyzing the algorithm described in Section~\ref{sec:the construction}. The proof is done in three steps, following the three steps of the construction. By Fact~\ref{prop:having lowdef t-source}, we may assume that $X$ has a $T_X$-structure and that $Y$ has a $T_Y$-structure for some entropy-trees $T_X, T_Y$, both with parameters $(k/2,2^{-\Omega(k^{1/10})})$. This costs only $\log{n}$ in deficiency and introduce a small error of $2^{-\Omega(k^{1/10})}$.  Throughout the analysis we only consider subsources with deficiency $o(k^{1/10})$ and so this error can be ignored.

\subsection{Analysis of Step 1}\label{sec:anal of step 1}

We start this section by proving the following claim.

\begin{claim}\label{claim:fixed first step}
There exist deficiency $\ell \log^2{n}$ subsources $\xf \subset X$, $\yf \subset Y$ with the following property. For every node $v$ in $T_X$ that is labeled by $\f$, it holds that $\ch((\xf)_v,\yf)$ is fixed to a constant. Further, for every node $u$ of $T_Y$ that is labeled by $\f$, it holds that $\ch((\yf)_u, \xf)$ is fixed to a constant.
\end{claim}

\begin{proof}
Let $v$ be a node in $T_X$ that is labeled by $\f$. Since $X$ has a $T_X$-structure, $X_v$ is fixed to a constant, and so $\ch(X_v, Y)$ is a deterministic function of $Y$. Since $\ch(X_v, Y)$ consists of $\ell \log{n}$ bits, Fact~\ref{fact:def and fixing} implies that there exists a deficiency $\ell \log{n}$-subsource $Y' \subset Y$ such that $\ch(X_v,Y')$ is fixed to a constant. Repeating this argument for every $v \in T_X$ that is labeled by $\f$, we get a subsource $\yf \subset Y$ such that $\ch(X_v,\yf)$ is fixed to a constant for every $v$ in $T_X$ that is labeled by $\f$. By the definition of an entropy-tree, there is at most one node labeled by $\f$ in each level of $T_X$ and since $\dep(T_X) = \log{n}$, we have that $\yf$ is a deficiency $\ell \log^2{n}$ subsource of $Y$.

Since $\yf$ is a subsource of $Y$, for every node $u$ in $T_Y$ that is labeled by $\f$ it holds that $(\yf)_v$ is fixed to a constant. We now perform the analogous process on $T_Y$ to obtain a deficiency $\ell \log^{2}{n}$ subsource $\xf \subset X$ such that for every node $u$ in $T_Y$ that is labeled by $\f$ it holds that $\ch((\yf)_u, \xf)$ is fixed to a constant. Note that since $\xf$ is a subsource of $X$, it also holds that $\ch((\xf)_v, \yf)$ is fixed to a constant for every $v$ in $T_X$ that is labeled by $\f$. Thus, informally speaking, by performing the analogous process to $T_Y$ we do not ``ruin'' the desired property we obtained first for $T_X$.
\end{proof}

%
%

Next we show that there exist low-deficiency subsources $\xfi \subset \xf$, $\yfi \subset \yf$ ($\mathsf{FI}$ stands for ``fixed identified''), restricted to which, $\subext$ correctly identifies the nodes in $T_X,T_Y$ that are labeled by $\f$.

\begin{claim}\label{claim:fixed second step}
There exist deficiency $O(\ell \log^2{n})$ subsources $\xfi \subset \xf$, $\yfi \subset \yf$ with the following property. For every node $v$ of $T_X$ that is labeled by $\f$ and for every node $u$ of $T_Y$ that is labeled by $\f$, it holds that
\begin{align*}
&\pr\left[ \parent(v) \,\, (\xfi,\yfi)\text{-favors }v\right] = 0,\\
&\pr\left[ \parent(u) \,\, (\xfi,\yfi)\text{-favors }u\right] = 0.
\end{align*}
\end{claim}

\begin{proof}

Let $v$ be a node in $T_X$ that is labeled by $\f$. We first note that by the definition of an entropy-tree, $\root(T_X)$ cannot be labeled by $\f$, and so it is valid to refer to $\parent(v)$. Further, by the definition of an entropy-tree, if a node is labeled by $\f$ then it must be the left son of its parent. Thus, $\parent(v)$ $(x,y)$-favors $v$ if and only if
\begin{equation}\label{eq:I like this dessert}
\resp\left( x_{\parent(v)}, y, \ch\left(x_{\leftson(\parent(v))}, y\right)\right) = \hasent.
\end{equation}

By Claim~\ref{claim:fixed first step}, $\ch((\xf)_v,\yf)$ is fixed to a constant. Thus, to apply the challenge-response mechanism, we only need to show that both $(\xf)_{\parent(v)}$ and $\yf$ have a sufficient amount of entropy. By the definition of an entropy-tree, since $\lbl(v) = \f$, $\lbl(\parent(v)) \in \{ \h, \btop, \bmid \}$ and so $X_{\parent(v)}$ is $\er$-close to having min-entropy $\Omega(k^{1/4})$. Since $\xf$ is a deficiency $O(\ell \log^{2}{n})$ subsource of $X$, Fact~\ref{fact:def and entropy} implies that $(\xf)_{\parent(v)}$ is $2^{-k^{\Omega(1)}}$-close to having min-entropy $\Omega(k^{1/4}) - O(\ell \log^{2}{n}) = \Omega(k^{1/4})$. By a similar argument, $\yf$ is $\er$-close to having min-entropy $\Omega(k)$. Since $k^{1/4} = \omega(\log^{10}{n})$, Theorem~\ref{thm:cr} implies that there exist deficiency $2\ell \log{n}$ subsources $X' \subset \xf$, $Y' \subset \yf$ such that for any $(x,y) \in \sup((X',Y'))$, Equation~\eqref{eq:I like this dessert} fails to hold. Thus, for any such $x,y$ it holds that $\parent(v)$ does not $(x,y)$-favor $v$.

We repeat this argument for every node $v$ in $T_X$ that is labeled by $\f$ and obtain deficiency $2 \ell \log^{2}{n}$ subsources $X'' \subset \xf$, $Y'' \subset \yf$ with the property that for every $v$ in $T_X$ that is labeled by $\f$ it holds that
$$
\pr
\left[
  \parent(v)\,\, (X'',Y'')\text{-favors }v
\right] = 0.
$$
This is possible since the entropy of $Y$ and the entropies of $X_v$ for $v$ labeled by one of $\{ \h, \btop, \bmid \}$ remain large enough throughout the process.

We now apply the same argument for every node $u$ in $T_Y$ that is labeled by $\f$. Since $X''$ and $Y''$ are deficiency $O(\ell \log^{2}{n})$ subsources of $\xf, \yf$, respectively, we can obtain deficiency $O(\ell \log^{2}{n}$) subsources $\xfi \subset \xf$, $\yfi \subset \yf$, such that for any node $u$ in $T_Y$ that is labeled by $\f$ it holds that
$$
\pr\left[ \parent(u) \,\, (\xfi,\yfi)\text{-favors }u\right] = 0.
$$
We note that since $\xfi$ and $\yfi$ are subsources of $X'',Y''$, it also holds that for every node $v$ in $T_X$ that is labeled by $\f$,
$$
\pr\left[ \parent(v) \,\, (\xfi,\yfi)\text{-favors }v\right] = 0.
$$
That is, we have not ``ruined'' the desired property we obtained first in $T_X$ when working on $T_Y$. This concludes the proof of the claim.
\end{proof}

Up to this point, we found low-deficiency subsources $\xfi \subset X$ and $\yfi \subset Y$ such that the nodes labeled by $\f$ in $T_X, T_Y$ are correctly identified by the challenge-response mechanism  when applied to samples from $\xfi,\yfi$. Next we prove that with high probability over $(\xfi,\yfi)$, the entropy-paths in $T_X,T_Y$ are identified correctly by the sub-extractor in the sense that the observed entropy-paths contain the entropy-paths of the respective entropy-trees.

\begin{claim}\label{claim:find paths}
Except with probability $2^{-\Omega(\ell)}$ over $(x,y) \sim (\xfi,\yfi)$, it holds that
\begin{align*}
&\forall i \in \{0,\ldots,\ibot(T_X)\} \quad v_i(x,y) = v_i(T_X), \\
&\forall i \in \{0,\ldots,\ibot(T_Y)\} \quad u_i(x,y) = u_i(T_Y).
\end{align*}
%
\end{claim}

\begin{proof}
We prove the first equation in the statement of the claim. The proof of the second equation is similar, and then the proof of the claim follows by the union bound. We first observe that by the definition of an entropy-tree, for any ancestor $v \neq \vvbot(T_X)$ of $\vvbot(T_X)$ it holds that $\lbl(\leftson(v)) = \f$ if and only if $\rightson(v)$ is the good son of $v$. Indeed, on one hand, if $\leftson(v)$ is labeled by $\f$ then $\leftson(v)$ cannot be an ancestor of $\vvbot(T_X)$ as all of $\leftson(v)$'s descendants have no label. On the other hand, since $v$ has a label and its label can only be one of $\h, \btop, \bmid$, if its left son is not labeled by $\f$ then $\rightson(v)$ has no label, and so $\rightson(v)$ cannot be an ancestor of $v$ as all of its descendants have no label.

Ideally, given this observation, we would have liked to prove by a backward induction on $i = \ibot(T_X)-1, \ldots, 1, 0$ that
$$
\pr_{(x,y) \sim (\xfi,\yfi)}
\left[
  \forall j \in \{i, \ldots, \ibot(T_X)-1\} \quad v_j(T_X) \,\, (x,y)\text{-favors its good son}\,
\right] \ge 1 - 2^{-\Omega(\ell)}.
$$
Indeed, note that the claim will then follow by considering $i = 0$. However, we need to prove a stronger statement so to have a stronger induction hypothesis, as otherwise we will not be able to carry the induction step. More precisely, set $t = 20 \ell \log{n}$. Let $\eps_{\ibot(T_X)-1} = 2^{-\Omega(\ell)}$. For $i = \ibot(T_X)-2, \ldots, 1, 0$, define $\eps_i = (2^{-\Omega(\ell)} + \eps_{i+1}) \cdot \poly(n)$. We prove by a backward induction on $i = \ibot(T_X)-1, \ldots, 1, 0$ that for any deficiency $it$ subsources $X' \subset \xfi$, $Y' \subset \yfi$, it holds that
$$
\pr_{(x,y) \sim (X',Y')}
\left[
  \forall j \in \{i, \ldots, \ibot(T_X)-1\} \quad v_j(T_X) \,\, (x,y)\text{-favors its good son}\,
\right] \ge 1 - \eps_i.
$$
We note that the claim follows by considering $i = 0$ as $\eps_0 = 2^{-\Omega(\ell)} \cdot 2^{O(\log^{2}{n})} = 2^{-\Omega(\ell)}$.

We start with the base of the induction $i = \ibot(T_X)-1$. Let $X' \subset \xfi$, $Y' \subset \yfi$ be deficiency $(\ibot(T_X)-1)t$ subsources. Consider two cases according to the label of $\leftson(v_{\ibot(T_X)-1})$. If $\lbl(\leftson(v_{\ibot(T_X)-1})) = \f$ then by Claim~\ref{claim:fixed second step},
$$
\pr_{(x,y) \sim (\xfi, \yfi)}
\left[
  v_{\ibot(T_X)-1}\,\, (x,y)\text{-favors } \leftson\left(v_{\ibot(T_X)-1}\right)
\right] = 0.
$$
Since $X',Y'$ are subsources of $\xfi, \yfi$, respectively, the same holds for $(x,y) \sim (X',Y')$. Moreover, as $\rightson(v_{\ibot(T_X)-1})$ is the good son of $v_{\ibot(T_X)-1}$, the basis of the induction for this case is proven.

Consider now the case $\lbl(\leftson(v_{\ibot(T_X)-1})) \ne \f$. By the observation above, in this case, the good son of $v_{\ibot(T_X)-1}$ is its left son and so $\leftson(v_{\ibot(T_X)-1}) = \vvbot(T_X)$. Thus, $v_{\ibot(T_X)-1}$ $(x,y)$-favors its good son if and only if
\begin{equation}\label{eq:i am tierd a bit}
\resp\left( x_{v_{\ibot(T_X)-1}}, y, \ch\left( x_{\vvbot(T_X)}, y \right) \right) = \hasent.
\end{equation}
Thus, to conclude the proof of the base case, it is enough to show that Equation~\eqref{eq:i am tierd a bit} holds with probability $1-2^{-\Omega(\ell)}$ over $(x,y) \sim (X',Y')$. To this end, first note that $\ch(x_{\vvbot(T_X)},y)$ contains $\bext(x_{\vvbot(T_X)},y)$ as a row. By Theorem~\ref{thm:cr}, it is enough to show that for all deficiency $t$ subsources $\hat{X} \subset X'$, $\hat{Y} \subset Y'$, it holds that $\bext(\hat{X}_{\vvbot(T_X)},\hat{Y})$ is close to uniform.

Since $\hat{X}$ is a deficiency $O(\ibot t + \ell \log^{2}{n}) = O(\ell \log^{2}{n})$ subsource of $X$ and since $X_{\vvbot(T_X)}$ is $\er$-close to an $\Omega(k^{1/8})$-block-source, $\hat{X}_{\vvbot(T_X)}$ is also $\er$-close to an $\Omega(k^{1/8})$-block-source. Further, since $\hat{Y}$ is a deficiency $O(\ell \log^{2}{n})$ subsource of $Y$, and since $\me(Y) \ge k$, $\me(\hat{Y}) = \Omega(k)$. Since $k^{1/8} = \omega(\log^{12}{n})$, Theorem~\ref{thm:li extractor} implies that $\bext(\hat{X}_{\vvbot(T_X)}, \hat{Y})$ is $\er$-close to a uniform string on $\ell$ bits. Thus, by Theorem~\ref{thm:cr}, Equation~\eqref{eq:i am tierd a bit} holds with probability at least $1 - 2^{-\Omega(\ell)}$.

We now proceed to the induction step. Let $0 \le i < \ibot(T_X)-1$. Let $X' \subset \xfi$, $Y' \subset \yfi$ be deficiency $it$ subsources. We want to show that
$$
\pr_{(x,y) \sim (X',Y')}
\left[
  \forall j \in \{i, \ldots, \ibot(T_X)-1\} \quad v_j(T_X) \text{ favors its good son}\,
\right] \ge 1 - \eps_i.
$$
By the induction hypothesis, for any deficiency $(i+1)t$ subsources $X'' \subset \xfi$, $Y'' \subset \yfi$, it holds that
$$
\pr_{(x,y) \sim (X'',Y'')}
\left[
  \forall j \in \{i+1, \ldots, \ibot(T_X)-1\} \quad v_j(T_X) \text{ favors its good son}\,
\right] \ge 1 - \eps_{i+1}.
$$
As was done in the basis of the induction, we consider two cases. If $\lbl(\leftson(v_i(T_X))) = \f$ then by Claim~\ref{claim:fixed second step},
$$
\pr_{(x,y) \sim (\xfi, \yfi)}
\left[
  v_i(T_X)\,\, (x,y)\text{-favors its good son}
\right] = 1.
$$
Since $X' \subset \xfi$ and $Y' \subset \yfi$, the same holds for $(x,y) \sim (X',Y')$. Thus, by the induction hypothesis
$$
\pr_{(x,y) \sim (X',Y')}
\left[
  \forall j \in \{i, \ldots, \ibot(T_X)-1\} \quad v_j(T_X) \text{ favors its good son}\,
\right]
\ge 1 - \eps_{i+1}
\ge 1 - \eps_{i}.
$$
Consider now the case $\lbl(\leftson(v_i(T_X))) \ne \f$. By the observation made at the beginning of the proof, in this case the good son of $v_i(T_X)$ is its left son. Thus, $v_i(T_X)$ $(x,y)$-favors its good son if and only if
\begin{equation}\label{eq:back at the office}
\resp\left( x_{v_i(T_X)}, y, \ch\left( x_{\leftson(v_i(T_X))}, y\right) \right) = \hasent.
\end{equation}

By Theorem~\ref{thm:cr}, it is enough to show that for any deficiency $t$ subsources $\hat{X} \subset X'$, $\hat{Y} \subset Y'$, it holds that $\ch(\hat{X}_{\leftson(v_i(T_X))},\hat{Y})$ is $\er$-close to having min-entropy $\ell$. Since $\hat{X}$ is a deficiency $t$ subsource of $X'$, and since $X'$ is a deficiency $it$ subsource of $\xfi$, we have that $\hat{X}$ is a deficiency $(i+1)t$ subsource of $\xfi$. Similarly, $\hat{Y}$ is a deficiency $(i+1)t$ subsource of $\yfi$. Thus, by the induction hypothesis,
$$
\pr_{(x,y) \sim (\hat{X},\hat{Y})}
\left[
  \forall j \in \{i+1, \ldots, \ibot(T_X)-1\} \quad v_j(T_X) \text{ favors its good son}\,
\right] \ge 1 - \eps_{i+1}.
$$
By the above equation and by the definition of $\ch$, except with probability $\eps_{i+1}$ over $(x,y) \sim (\hat{X},\hat{Y})$, it holds that $\bext(x_{\vvbot(T_X)}, y)$ appears as a row in $\ch(x_{v_{i+1}(T_X)},y)$.

Since $\hat{X}$ is a deficiency $O((i+1)t + \ell \log^{2}{n}) = O(\ell \log^{2}{n})$ subsource of $X$ and since $X_{\vvbot(T_X)}$ is $\er$-close to an $\Omega(k^{1/8})$-block-source, $\hat{X}_{\vvbot(T_X)}$ is also $\er$-close to an $\Omega(k^{1/8})$-block-source. Further, since $\hat{Y}$ is a deficiency $O(\ell \log^{2}{n})$ subsource of $Y$ and since $\me(Y) \ge k$, $\me(\hat{Y}) = \Omega(k)$. As we assume $k^{1/8} = \omega(\log^{12}{n})$, Theorem~\ref{thm:li extractor} implies that $\bext(\hat{X}_{\vvbot(T_X)}, \hat{Y})$ is $\er$-close to uniform. Thus, except with probability $\eps_{i+1} + \er$, $\ch(\hat{X}_{\leftson(v_i(T_X))},\hat{Y})$ has min-entropy $\ell$. Thus, by Theorem~\ref{thm:cr}, Equation~\eqref{eq:back at the office} holds except with probability $1-(2^{-\Omega(\ell)} + \eps_{i+1}) \cdot \poly(n) = 1 - \eps_i$.
This concludes the proof of the claim.
\end{proof}

\subsection{Analysis of Step 2}

Informally speaking, in this section we prove that the sub-extractor correctly identifies $\vvmid(T_X)$ in some carefully chosen subsources of $\xfi,\yfi$. More precisely, we would have wanted to prove a statement of the following form:
\paragraph{A wishful claim.}There exist low-deficiency subsources $X' \subset \xfi$, $Y' \subset \yfi$ such that with high probability over $(x,y) \sim (X',Y')$, $\obsvvmid(x,y) = \vvmid(T_X)$.

\medskip
Unfortunately, we will not be able to prove this statement. We will, however, be able to prove the same statement for $X',Y'$ that have \emph{high-deficiency} in $\xfi,\yfi$. Still, $X',Y'$ will have enough entropy and structure so to carry out the rest of the analysis. Furthermore, the error term that we are carrying will not cause any harm even after moving to these high-deficiency subsources.



For $\alpha \in \sup((\xfi)_{\leftson(\vvmid(T_X))})$ and $\beta \in \sup((\yfi)_{\leftson(\uumid(T_Y))})$, we define
\begin{align*}
X_{\alpha} &= \xfi \mid ((\xfi)_{\leftson(\vvmid(T_X))} = \alpha), \\
Y_{\beta} &= \yfi \mid ((\yfi)_{\leftson(\uumid(T_Y))} = \beta).
\end{align*}
Let $B$ be the set of all $(x,y) \in \sup((\xfi,\yfi))$ such that
\begin{align}\label{eq:not in b}
&\exists i \in \{0,\ldots,\ibot(T_X)\} \quad v_i(x,y) \neq v_i(T_X) \quad \vee\notag \\
&\exists i \in \{0,\ldots,\ibot(T_Y)\} \quad u_i(x,y) \neq u_i(T_Y).
\end{align}
By Claim~\ref{claim:find paths},
$$
\pr[(\xfi,\yfi) \in B] \le 2^{-\Omega(\ell)}.
$$
Thus, by averaging, there exist $\alpha,\beta$ such that
$$
\pr[(X_{\alpha},Y_{\beta}) \in B] \le
2^{-\Omega(\ell)}.
$$
These are the subsources $X_\alpha \subset \xfi$, $Y_\beta \subset \yfi$ that we will work with. We think of $(x,y) \in B$ as an ``error'' and ignore this event for now. We later accumulate the error coming from this event while making sure to treat the error correctly when moving into subsources of $(X_{\alpha}, Y_{\beta})$. More precisely, recall that by moving to deficiency $d$ subsource, an error of $\eps$ in the source can ``grow'' to at most $2^{d} \cdot \eps$ restricted to the subsource. Since the error term is $2^{-\Omega(\ell)}$ and since we will move to deficiency $o(\ell)$-subsources, the error will remain $2^{-\Omega(\ell)}$ in the subsources that we will restrict to. Thus, we assume that Equation~\eqref{eq:not in b} holds. In particular, we assume that $v_{\itop(T_X)}(X_\alpha,Y_\beta) = \vvtop(T_X)$, $v_{\imid(T_X)}(X_\alpha,Y_\beta) = \vvmid(T_X)$, etc.


\medskip

Recall that $\obsvvmid(x,y)$ is defined to be the node $v$ in $\pobs(x,y)$, with the largest depth such that
\begin{equation}\label{eq:test for vvmid}
\resp\left( x, y, \newch\left( x_v, y_{\qobs(x,y)} \right) \right) = \hasent.
\end{equation}
Thus, to show that $\vvmid(T_X)$ is correctly identified on low-deficiency subsources of $(X_{\alpha}, Y_{\beta})$, we first show that there exist low-deficiency subsources $X_{\alpha,\beta} \subset X_{\alpha}$, $Y_{\alpha,\beta} \subset Y_{\beta}$ such that with high probability over $(x,y) \sim (X_{\alpha,\beta}, Y_{\alpha,\beta})$, Equation~\eqref{eq:test for vvmid} does not hold with $v = v_i(x,y)$ for all $i > \imid(T_X)$. This is the content of the following claim. Afterwards, in Claim~\ref{claim:at imid}, we show that with high probability over $(x,y) \sim (X_{\alpha,\beta}, Y_{\alpha,\beta})$, Equation~\eqref{eq:test for vvmid} holds with $v = v_{\imid(T_X)}(x,y) = \vvmid(T_X)$. 

\begin{claim}\label{claim:below imid}
There exist deficiency $O(\ell' \log^{2}{n})$-subsources $X_{\alpha,\beta} \subset X_{\alpha}$, $Y_{\alpha,\beta} \subset Y_{\beta}$ such that with probability $1-2^{-\Omega(\ell)}$ over $(x,y) \sim (X_{\alpha,\beta}, Y_{\alpha,\beta})$, it holds that
\begin{equation}\label{eq:I want to end this already}
\forall i > \imid(T_X) \quad \resp\left( x, y, \newch\left( x_{v_i(x,y)}, y_{\qobs(x,y)} \right) \right) = \fixed.
\end{equation}
%
\end{claim}

Towards proving Claim~\ref{claim:below imid}, we start by proving the following two claims.

\begin{claim}\label{claim:qobs is fixed}
There exists a deficiency $\log{n}$-subsource $X'_{\alpha} \subset X_{\alpha}$ such that $\qobs(X'_{\alpha},Y_\beta)$ is fixed to a constant.
\end{claim}

\begin{proof}
Recall that $\qobs(X_{\alpha},Y_\beta)$ is the path
$$
u_0(X_{\alpha},Y_\beta), \ldots,
u_{\ibot(T_Y)}(X_{\alpha},Y_\beta),
u_{\ibot(T_Y)+1}(X_{\alpha},Y_\beta), \ldots,
u_{\log(n)-1}(X_{\alpha},Y_\beta).
$$
By Equation~\eqref{eq:not in b}, for $i \le \ibot(T_Y)$ it holds that $u_i(X_\alpha, Y_\beta) = u_i(T_Y)$, and so for such $i$, $u_i(X_\alpha, Y_\beta)$ is fixed to a constant. We now consider the case $i > \ibot(T_Y)$. Consider first the case $i = \ibot(T_Y)+1$. In this case, $u_i(X_\alpha,Y_\beta)$ determines which of the two sons of $u_{\ibot(T_Y)}(x,y) = \uubot(T_Y)$ is on the observed entropy-path $\qobs(X_\alpha,Y_\beta)$. Recall that this decision is based on whether or not
\begin{equation}\label{eq:ch vs resp}
\resp\left( (Y_\beta)_{\uubot(T_Y)}, X_\alpha, \ch\left( (Y_\beta)_{\leftson(\uubot(T_Y))}, X_\alpha \right) \right) = \hasent.
%
\end{equation}
Since $\uubot(T_Y)$ and $\leftson(\uubot(T_Y))$ are descendants of $\leftson(\uumid(T_Y))$, as follows by the definition of an entropy-tree, and since $(Y_\beta)_{\leftson(\uumid(T_Y))}$ is fixed to $\beta$, it also holds that $(Y_\beta)_{\uubot(T_Y)}$ and $(Y_\beta)_{\leftson(\uubot(T_Y))}$ are fixed to constants. Thus, Equation~\eqref{eq:ch vs resp} is determined only by $X_\alpha$. Since Equation~\eqref{eq:ch vs resp} gives one bit of information on $X_\alpha$, by Fact~\ref{fact:def and fixing}, there exists a deficiency $1$ subsource $X' \subset X_\alpha$ such that $u_i(X',Y_\beta)$ is fixed to a constant.

We now repeat this argument for $i = \ibot(T_Y)+2, \ldots, \log(n)-1$. In each iteration we make sure that the next descendant of $\uubot(T_Y)$ is fixed to a constant on a low-deficiency subsource of $X_\alpha$ with $Y_\beta$. Since we repeat this process for at most $\log{n}$ iterations, we will eventually obtain a deficiency $\log{n}$-subsource $X'_\alpha \subset X_\alpha$ such that $\qobs(X'_\alpha,Y_\beta)$ is fixed to a constant, as desired. Accounting back for the error, note that since $\ell = \omega(\log{n})$, it holds that $\pr[(X'_\alpha, Y_\beta) \in B] \le 2^{-\Omega(\ell)}$.
\end{proof}

Claim~\ref{claim:qobs is fixed} allows us to set the ground for the challenge-response mechanism:

\begin{claim}\label{claim:the heart of it all}
For any $i > \imid(T_X)$,
$$
\newch\left((X'_{\alpha})_{v_i(X'_{\alpha},Y_\beta)}, (Y_\beta)_{\qobs(X'_{\alpha},Y_\beta)}\right)
$$
is a deterministic function of $Y_\beta$.
\end{claim}

\begin{proof}
By Claim~\ref{claim:qobs is fixed} we have that $\qobs(X'_{\alpha},Y_\beta)$ is fixed to a constant. Thus, it suffices to show that $(X'_{\alpha})_{v_i(X'_{\alpha},Y_\beta)}$ is a deterministic function of $Y_\beta$.

We start by considering the case $i = \imid(T_X)+1$. In this case, $v_i(X'_{\alpha},Y_\beta)$ is fixed to a constant. Indeed, since $i = \imid(T_X)+1 \le \ibot(T_X)$, it holds by Equation~\eqref{eq:not in b} that $v_i(X_{\alpha},Y_{\beta}) = v_i(T_X)$ and so, since $X'_\alpha$ is a subsource of $X_\alpha$, $v_i(X'_{\alpha},Y_{\beta}) = v_i(T_X)$.

The case $i > \imid(T_X)+1$ follows by a different logic, similar to that used in the proof of Claim~\ref{claim:qobs is fixed}. Lets first consider the case $i = \imid(T_X)+2$. Recall that $(X_\alpha)_{\leftson(\vvmid(T_X))}$ is fixed to a constant. Thus, also $(X'_{\alpha})_{\leftson(\vvmid(T_X))}$ is fixed to a constant. Now, $v_i(X'_{\alpha},Y_\beta)$ is defined to be one of the two sons of $\leftson(\vvmid(T_X))$ according to the Boolean value of the expression
$$
\resp
  \left(
    (X'_{\alpha})_{\leftson(\vvmid(T_X))}, Y_\beta, \ch \left( (X'_\alpha)_{\leftson(\leftson(\vvmid(T_X)))}, Y_\beta \right)
  \right) = \hasent.
$$
Since $(X'_{\alpha})_{\leftson(\vvmid(T_X))}$ is fixed to a constant, the above equation is determined only by $Y_\beta$. This shows that $v_i(X'_{\alpha}, Y_\beta)$ is a deterministic function of $Y_\beta$ for $i = \imid(T_X)+2$. A similar argument can be used to show that the same holds for any $i > \imid(T_X)+1$.

To conclude the proof of the claim, we need to show that $(X'_{\alpha})_{v_i(X'_{\alpha}, Y_\beta)}$ is a deterministic function of $Y_\beta$ for $i > \imid(T_X)$. Recall that for any such $i$, $v_i(X'_{\alpha}, Y_\beta)$ is a descendant of $\leftson(\vvmid(T_X))$ determined only by $Y_\beta$. The claim then follows as $(X'_{\alpha})_{\leftson(\vvmid(T_X))}$ is fixed to a constant.
\end{proof}

We are now ready to prove Claim~\ref{claim:below imid}.

\begin{proof}[Proof of Claim~\ref{claim:below imid}]
By Claim~\ref{claim:the heart of it all},
$$
\newch((X'_{\alpha})_{v_i(X'_{\alpha},Y_\beta)}, (Y_\beta)_{\qobs(X'_{\alpha},Y_\beta)})
$$
is a deterministic function of $Y_\beta$ for all $i > \imid(T_X)$. Thus, there exists a deficiency $O(\ell' \log^2{n}) = o(\ell)$ subsource $Y'_\beta \subset Y_\beta$ such that for all $i > \imid(T_X)$,
$$
\newch((X'_{\alpha})_{v_i(X'_{\alpha},Y'_\beta)}, (Y'_\beta)_{\qobs(X'_{\alpha},Y'_\beta)})
$$
is fixed to a constant. By Theorem~\ref{thm:cr}, and accounting for the error, there exist deficiency $2\ell' \log^2{n} = o(\ell)$ subsources $X_{\alpha,\beta} \subset X'_{\alpha}$, $Y'_{\alpha,\beta} \subset Y'_\beta$, such that with probability $1-2^{-\Omega(\ell)}$ over $(x,y) \sim (X_{\alpha,\beta}, Y_{\alpha,\beta})$ Equation~\eqref{eq:I want to end this already} holds.

We note that this application of Theorem~\ref{thm:cr} is valid as both $X_{\alpha,\beta}$, $Y_{\alpha,\beta}$ are $\er$-close to having min-entropy $\Omega(\log^{10}{n})$. Indeed, $X_{\alpha,\beta}$ is a deficiency $O(\ell' \log^{2}{n})$-subsource of $X_\alpha = \xfi \mid ((\xfi)_{\leftson(\vvmid(T_X))} = \alpha)$. Since $\xfi$ is a deficiency $O(\ell \log^{2}{n})$-subsource of $X$ and since $X_{\vvmid(T_X)}$ is $\er$-close to an $\Omega(k^{1/4})$-block-source, it holds that $(\xfi)_{\vvmid(T_X)}$ is also $\er$-close to an $\Omega(k^{1/4})$-block-source. Thus, $X_\alpha$ is $\er$-close to having min-entropy $\Omega(k^{1/4})$. Therefore, $X_{\alpha,\beta}$ is $\er$-close to having min-entropy $\Omega(\log^{10}{n})$. A similar argument can be used for $Y_{\alpha,\beta}$.
\end{proof}

\begin{claim}\label{claim:at imid}
With probability $1- 2^{-\Omega(\ell')}$ over $(x,y) \sim (X_{\alpha, \beta},Y_{\alpha, \beta})$, it holds that
\begin{equation}\label{eq:i love yahli}
\resp\left(
x,y,
\newch
  \left(
    x_{v_{\imid(T_X)}(x,y)},
    y_{\qobs(x,y)}
  \right)
\right) = \hasent.
\end{equation}
\end{claim}

\begin{proof}
Let $B$ be the event defined in Equation~\eqref{eq:not in b}. As usual, we consider $(x,y) \in B$ as an ``error'' and ignore it for now. In particular, $v_{\imid(T_X)}(x,y) = \vvmid(T_X)$ and the path $\qobs(x,y)$ is assumed to contain $\uutop(T_Y)$. Thus, $\newch\left(x_{v_{\imid(T_X)}(x,y)}, y_{\qobs(x,y)}\right)$ contains $\bext\left(y_{\uutop(T_Y)}, x_{\vvmid(T_X)}\right)$ as a row.

Recall that $(Y)_{\uutop(T_Y)}$ is $2^{-\Omega(k^{1/10})}$-close to an $\Omega(k^{1/2})$-block-source. Since $\yfi$ is a deficiency $O(\ell \log^{2}{n})$-subsource of $Y$, and since $\ell \log^{2}{n} = o(k^{1/10})$, $(\yfi)_{\uutop(T_Y)}$ is also $2^{-\Omega(k^{1/10})}$-close to an $\Omega(k^{1/2})$-block-source. By a similar argument, $(\yfi)_{\uumid(T_Y)}$ is $2^{-\Omega(k^{1/10})}$-close to an $\Omega(k^{1/4})$-block-source. Thus, $(Y_{\beta})_{\uumid(T_Y)}$ is $2^{-\Omega(k^{1/10})}$-close to having min-entropy $\Omega(k^{1/4})$, and so $(Y_{\beta})_{\uutop(T_Y)}$ is $2^{-\Omega(k^{1/10})}$-close to an $\Omega(k^{1/4})$-block-source. Therefore, $(Y_{\alpha,\beta})_{\uutop(T_Y)}$ $2^{-\Omega(k^{1/10})}$-close to an $\Omega(k^{1/4})$-block-source.

Recall that $X_{\vvmid(T_X)}$ is $2^{-\Omega(k^{1/10})}$-close to an $\Omega(k^{1/4})$-block-source. Since $\xfi$ is a deficiency $O(\ell \log^{2}{n})$-subsource of $X$, $(\xfi)_{\vvmid(T_X)}$ is $2^{-\Omega(k^{1/10})}$-close to an $\Omega(k^{1/4})$-block-source. Thus, $(X_{\alpha})_{\vvmid(T_X)}$ is $2^{-\Omega(k^{1/10})}$-close to having min-entropy $\Omega(k^{1/4})$. Therefore, $(X_{\alpha,\beta})_{\vvmid(T_X)}$ is $2^{-\Omega(k^{1/10})}$-close to having min-entropy $\Omega(k^{1/4})$.

Let $\hat{X} \subset X_{\alpha,\beta}$, $\hat{Y} \subset Y_{\alpha,\beta}$ be any deficiency $20 \ell'  \log{n}$ subsources. We have that $\hat{Y}_{\uutop(T_Y)}$ is $2^{-\Omega(k^{1/10})}$-close to an $\Omega(k^{1/4})$-block-source and that $\hat{X}_{\vvmid(T_X)}$ is $2^{-\Omega(k^{1/10})}$-close to having min-entropy $\Omega(k^{1/4})$. Thus, $\bext(\hat{Y}_{\uutop(T_Y)}, \hat{X}_{\vvmid(T_X)})$ is $\er$-close to a uniform string on $\ell'$ bits. In particular, $\newch(\hat{X}_{\vvmid(T_X)}, \hat{Y}_{\qobs(\hat{X},\hat{Y})})$ is $\er$-close to having min-entropy $\ell'$.

Theorem~\ref{thm:cr} then implies that Equation~\eqref{eq:i love yahli} holds with probability at least $1 - (2^{-\ell'}+\er) = 1 - 2^{-\Omega(\ell')}$ over $(x,y) \sim (X_{\alpha,\beta}, Y_{\alpha,\beta})$. This concludes the proof of the claim.

\end{proof}

\subsection{Analysis of Step 3}

Recall that the output of the sub-extractor is defined as
$$
\subext(x,y) = \bext\left(x_{\obsvvmid(x,y)} \circ x, y\right).
$$
By Claim~\ref{claim:below imid} and Claim~\ref{claim:at imid}, we have that except with probability $2^{-\Omega(\ell')}$ over $(x,y) \sim (X_{\alpha,\beta}, Y_{\alpha,\beta})$ it holds that $\obsvvmid(x,y) = \vvmid(T_X)$. Recall that $\vvmid(T_X)$ is a descendant of $\leftson(\vvtop(T_X))$. Further, recall that $(X_{\alpha,\beta})_{\vvmid(T_X)}$ is $\er$-close to having min-entropy $\Omega(k^{1/4})$. Since $X_{\vvtop(T_X)}$ is $2^{-\Omega(k^{1/10})}$-close to a $\Omega(k^{1/2})$-block-source, we have that $(X_{\alpha,\beta})_{\vvtop(T_X)}$ is $2^{-\Omega(k^{1/10})}$-close to an $\Omega(k^{1/4})$-block-source. In particular, this implies that $(X_{\alpha,\beta})_{\vvmid(T_X)} \circ X_{\alpha,\beta}$ is $\er$-close to an $\Omega(k^{1/4})$-block-source.

Recall also that $Y_{\alpha,\beta}$ is $2^{-\Omega(k^{1/10})}$-close to having min-entropy $\Omega(k)$. Thus, by Theorem~\ref{thm:li extractor} we conclude that
$$
\subext(X_{\alpha,\beta},Y_{\alpha,\beta}) =
\bext\left((X_{\alpha,\beta})_{\obsvvmid(X_{\alpha,\beta},Y_{\alpha,\beta})} \circ X_{\alpha,\beta}, Y_{\alpha,\beta}\right)
$$
is $\er$-close to uniform.

\section{Conclusion and Open Problems}\label{sec:open}

\paragraph{The next quantitative natural goal.}In this paper, we gave a construction of a $2^{\poly(\log\log{n})}$-Ramsey graph, or equivalently, a two-source disperser for entropy $\polylog(n)$. Erd{\"o}s set the goal at constructing $O(\log{n})$-Ramsey graphs, which translates to the difficult problem of constructing two-source dispersers for entropy $\log(n)+O(1)$. We set the next goal towards Erd{\"o}s challenge at constructing a $\polylog(n) = 2^{O(\log\log{n})}$-Ramsey graph, which is equivalent to a two-source disperser for entropy $O(\log{n})$.

\paragraph{A weakly-explicit construction.}Our construction of Ramsey graphs is strongly-explicit, namely, one can query each pair of vertices of the $n$-vertices graph to check whether there is an edge connecting them, in time $\polylog(n)$. In the setting of two-source dispersers, a strongly-explicit construction is the natural definition. However, we believe it is interesting to obtain better weakly-explicit Ramsey graphs, where by weakly-explicit we mean that the entire graph can be computed in time $\poly(n)$. Barak~\etal~\cite{BKSSW10} have a simple construction of a $\polylog(n)$-Ramsey graph, however, its running-time is $2^{\polylog{n}}$. Other than that, we are not aware of any result in this direction.

\paragraph{Improved sub-extractors.}The two-source sub-extractor that we construct has inner-entropy $\kout^{\Omega(1)}$ or even $\kout/\polylog(n)$, where $\kout$ is the outer-entropy. We pose the problem of constructing a sub-extractor with inner-entropy $\kin = \Omega(\kout)$ or even $\kin = \kout-o(\kout)$, for $\kout = \polylog(n)$. We believe that this is a natural goal towards constructing two-source extractors for polylogarithmic entropy.

\paragraph{Affine dispersers.}Shaltiel~\cite{S11} adjusted the challenge-response mechanism so to work with a single affine-source, rather than with two weak-sources. This allowed him to construct affine-dispersers for entropy $2^{(\log{n})^{0.9}}$. Can one obtain affine-dispersers for polylogarithmic entropy given recent advances?

\section*{Acknowledgement}\label{sec:ack}

I wish to thank Ran Raz and Avi Wigderson for their warm encouragement. On a personal note, it is uncustomary to acknowledge one's partner in life in mathematical papers. However, given that this paper was intensively written in the last month of my wife's pregnancy and in the first month of parenthood to the newborn baby girl Meshi and to our sweet Yahli, I will allow myself to make an exception -- thank you Orit! Your support and belief in me are uncanny.

\bibliographystyle{alpha}
\bibliography{bibliography}

\newcommand{\etalchar}[1]{$^{#1}$}
\begin{thebibliography}{BRSW12}

\bibitem[Abb72]{abbott72}
H.~L. Abbott.
\newblock Lower bounds for some {R}amsey numbers.
\newblock {\em Discrete Mathematics}, 2(4):289--293, 1972.

\bibitem[Alo98]{Alon1998shannon}
N.~Alon.
\newblock The shannon capacity of a union.
\newblock {\em Combinatorica}, 18(3):301--310, 1998.

\bibitem[Bar06]{barak06ramsey}
B.~Barak.
\newblock A simple explicit construction of an
  $n^{\tilde{o}(\log{n})}$-{R}amsey graph.
\newblock {\em arXiv preprint math/0601651}, 2006.

\bibitem[BKS{\etalchar{+}}10]{BKSSW10}
Boaz Barak, Guy Kindler, Ronen Shaltiel, Benny Sudakov, and Avi Wigderson.
\newblock Simulating independence: New constructions of condensers, {R}amsey
  graphs, dispersers, and extractors.
\newblock {\em Journal of the ACM (JACM)}, 57(4):20, 2010.

\bibitem[Bou05]{Bourgain05}
J.~Bourgain.
\newblock More on the sum-product phenomenon in prime fields and its
  applications.
\newblock {\em International Journal of Number Theory}, 1(01):1--32, 2005.

\bibitem[BRSW12]{brsw06}
B.~Barak, A.~Rao, R.~Shaltiel, and A.~Wigderson.
\newblock 2-source dispersers for $n^{o(1)}$ entropy, and {R}amsey graphs
  beating the {F}rankl-{W}ilson construction.
\newblock {\em Annals of Mathematics}, 176(3):1483--1544, 2012.

\bibitem[CG88]{CG88}
B.~Chor and O.~Goldreich.
\newblock Unbiased bits from sources of weak randomness and probabilistic
  communication complexity.
\newblock {\em SIAM Journal on Computing}, 17(2):230--261, 1988.

\bibitem[Chu81]{chung1981}
F.R.K. Chung.
\newblock A note on constructive methods for {R}amsey numbers.
\newblock {\em Journal of Graph Theory}, 5(1):109--113, 1981.

\bibitem[Erd47]{Erdos47}
P.~Erd{\"o}s.
\newblock Some remarks on the theory of graphs.
\newblock {\em Bulletin of the American Mathematical Society}, 53(4):292--294,
  1947.

\bibitem[Fra77]{frankl1977}
P.~Frankl.
\newblock A constructive lower bound for some {R}amsey numbers.
\newblock {\em Ars Combinatoria}, 3:297--302, 1977.

\bibitem[FW81]{FW81}
P.~Frankl and R.~M. Wilson.
\newblock Intersection theorems with geometric consequences.
\newblock {\em Combinatorica}, 1(4):357--368, 1981.

\bibitem[Gro01]{grolmusz2001low}
V.~Grolmusz.
\newblock Low rank co-diagonal matrices and {R}amsey graphs.
\newblock {\em Journal of combinatorics}, 7(1):R15--R15, 2001.

\bibitem[Li15]{Li15}
X.~Li.
\newblock Three-source extractors for polylogarithmic min-entropy.
\newblock {\em Electronic Colloquium on Computational Complexity (ECCC)}, 2015.

\bibitem[Nag75]{nagy1975}
Zs. Nagy.
\newblock A constructive estimation of the {R}amsey numbers.
\newblock {\em Mat. Lapok}, 23:301--302, 1975.

\bibitem[Nao92]{naor1992constructing}
M.~Naor.
\newblock Constructing {R}amsey graphs from small probability spaces.
\newblock {\em IBM Research Report RJ 8810}, 1992.

\bibitem[PR04]{PR04}
P.~Pudl{\'a}k and V.~R{\"o}dl.
\newblock Pseudorandom sets and explicit constructions of {R}amsey graphs.
\newblock {\em Quad. Mat}, 13:327–--346, 2004.

\bibitem[Ram28]{Ramsey1928}
F.~P. Ramsey.
\newblock On a problem of formal logic.
\newblock {\em Proceedings of the London Mathematical Society}, 30(4):338--384,
  1928.

\bibitem[Raz05]{Raz05}
R.~Raz.
\newblock Extractors with weak random seeds.
\newblock In {\em Proceedings of the thirty-seventh annual {ACM} Symposium on
  Theory of Computing}, pages 11--20, 2005.

\bibitem[Sha11]{S11}
R.~Shaltiel.
\newblock Dispersers for affine sources with sub-polynomial entropy.
\newblock In {\em Foundations of Computer Science (FOCS), 2011 IEEE 52nd Annual
  Symposium on}, pages 247--256. IEEE, 2011.

\end{thebibliography}


\end{document}